\documentclass[leqno]{amsart}
\author{Miljan Brako\v cevi\'c}
\title{Non-vanishing modulo $p$ of central critical Rankin--Selberg L-values with anticyclotomic twists}
\date{\today}
\usepackage{amsmath,amsthm,amssymb,amscd}
\usepackage{mathrsfs}
\usepackage{indentfirst}
\usepackage{url}
\usepackage{fullpage}

\newcommand{\X}[1]{\ensuremath{X(#1)}}
\newcommand{\Xw}[2]{\ensuremath{X(#1)_{/#2}}}
\newcommand{\xunrw}[2]{\ensuremath{\left(X(#1),\eta^{(p)}(#1)\right)_{/#2}}}
\newcommand{\ord}[1]{\ensuremath{\eta_p^{\mathrm{ord}}(#1)}}
\newcommand{\et}[1]{\ensuremath{\eta_p^{\mathrm{\acute{e}t}}(#1)}}
\newcommand{\xw}[2]{\ensuremath{\left(X(#1),\eta^{(p)}(#1),\et{#1}\times\ord{#1}\right)_{/#2}}}
	
\numberwithin{equation}{section}

\newtheorem{main}{Theorem}[section]
\newtheorem*{qexp}{Theorem}
\newtheorem{switch}{Remark}[section]
\newtheorem{factoring}{Lemma}[section]
\newtheorem{improv}[factoring]{Lemma}
\newtheorem*{recipe}{Recipe for a choice of $\lambda$}
\newtheorem{fact}{Fact}[section]
\newtheorem{density}[fact]{Theorem}
\newtheorem*{chai}{Theorem}
\newtheorem{linebundle}[fact]{Corollary}
\newtheorem{isogmaneuver}{Lemma}[section]
\newtheorem{assumption}[isogmaneuver]{Assumption}
\newtheorem{spcl}[isogmaneuver]{Proposition}

\begin{document}
\address{Department of Mathematics, UCLA, Los Angeles, CA 90095-1555, USA} 
\email{miljan@math.ucla.edu}
\subjclass[2010]{11F67}
\keywords{Non-vanishing, Rankin--Selberg L-values; CM points; Shimura curves}
\thanks{This research work is partially supported by Prof. Haruzo Hida's NSF grant DMS-0753991
through graduate student research fellowship and by UCLA Dissertation Year Fellowship.}

\begin{abstract}
We prove non-vanishing modulo $p$, for a prime $\ell \not = p$, of central critical Rankin--Selberg L-values with anticyclotomic twists of $\ell$-power conductor. The L-function is Rankin product of a cusp form and a theta series of arithmetic Hecke character of an imaginary quadratic field.  The paper is concerned with the case when the weight of Hecke character is greater than that of cusp form, so the L-value is essentially different in nature from the one in the landmark work of Vatsal and Cornut--Vatsal on the same theme.
\end{abstract}
\maketitle

\section{Introduction}

Studying non-vanishing of the central critical values of modular L-functions has had powerful and far-reaching applications to important problems of Iwasawa theory including various proofs of Main Conjectures. Generalizing the method of Sinnott (\cite{Si}) to the context of the theory of Shimura varieties, Hida studied non-vanishing of Hecke L-functions of CM fields in \cite{HidaDwork} and \cite{HidaLMS} and computed the $\mu$-invariant of Katz $p$-adic L-function in \cite{minv}. In his dissertation \cite{HS}, written under Hida's supervision in 2007, Hae-Sang Sun used this method to prove non-vanishing modulo (a rational prime) $p$ of L-values of the modular L-function associated to a level 1 Hecke newform $f$ twisted by a product $\lambda\chi$ of a fixed arithmetic Hecke character $\lambda$ of an imaginary quadratic field and finite order anticyclotomic $\chi$'s of $\ell$-power conductor, for a prime $\ell\not = p$. Here a fixed $\lambda$ is of $\infty$-type $(k,0)$, where $k$ is a weight of $f$.

Following the path paved by \cite{HidaDwork} and \cite{HS}, the purpose of this paper is to extend such non-vanishing result to the case of a Hecke newform of an arbitrary level $N\geq 1$ and nebentypus $\psi$, and arithmetic Hecke characters $\lambda$ of $\infty$-type $(k+m,-m)$ for arbitrary $m\geq0$. The moral of the latter is that, if we fix a Hecke character $\lambda_0$ of $\infty$-type $(k,0)$ satisfying certain criticality condition stated below, the anticyclotomic twists in our case are actually of the form $\chi_0\chi$ where $\chi_0$ is a fixed anticyclotomic Hecke character of $\infty$-type $(m,-m)$ and $\chi$'s range through the family of finite order anticyclotomic characters of $\ell$-power conductor -- in other words, our $\lambda:=\lambda_0\chi_0$.

The main ingredients in the proof are Zarisky density of CM points on modular Shimura varieties  studied in \cite{HidaDwork} and \cite{minv}, and the recent computation of an explicit Waldspurger formula in \cite{HidaNV}. Our modest aim is to understand the passage between these deep works of quite different mathematical flavor. In upcoming companion paper \cite{Br1} we compute the $\mu$-invariant of anticyclotomic $p$-adic L-function constructed in \cite{Br2}. 

To state the main theorem precisely, let $M$ be an imaginary quadratic field of discriminant $d(M)$, and set $d:=|d(M)|$. Fix two rational primes $p\not = \ell$  such that $p$ splits in $M$. Let $f$  be a normalized Hecke newform of level $\Gamma_0(N)$, $N\geq 1$, weight $k\geq1$, and nebentypus $\psi$ and let $\mathbf{f}$ be its corresponding adelic form on  $G(\mathbb{Q})\backslash G(\mathbb{A})$ with central character $\boldsymbol{\psi}$ (see Section \ref{sec:cmfadel} for definiton). All reasonable adelic lifts of $f$ are equal up to twists by a power of the everywhere unramified character $|\mathrm{det}(g)|_{\mathbb{A}}$ and $\mathbf{f}^u(g) := \mathbf{f}|\boldsymbol{\psi}(\mathrm{det}(g))|^{-1/2}$ is a unique one which generates a unitary automorphic representation  $\pi_\mathbf{f}$. We further take the base-change  $\hat{\pi}_{\mathbf{f}}$ to $\mathrm{Res}_{M/\mathbb{Q}}G$. Pick an arithmetic Hecke character $\lambda$ of $M^{\times} \backslash M^{\times}_{\mathbb{A}}$ of $\infty$-type $(k+m,-m)$ for arbitrary $m\geq 0$ and such that condition $\lambda|_{\mathbb{A}^\times} = \boldsymbol{\psi}^{-1}$ holds. Under this condition, the L-value $L(1/2,\hat{\pi}_\mathbf{f}\otimes \lambda^-)$, regarded as that of the Rankin--Selberg L-function associated to $\mathbf{f}$ and the theta series $\theta(\lambda^-)$ of $\lambda^-$, is critical in the sense of Deligne, and central with respect to the functional equation.

The choice of $\lambda$ is subtle and a detailed recipe is provided in Section \ref{sec:hewf}. Let $N=\prod_ll^{\nu(l)}$ be the prime factorization and denote by $N_{ns}=\prod_{l \, \text{non-split}}l^{\nu(l)}$ its ``non-split'' part. We choose $\lambda_0$ and $\chi_0$ as above requiring that $\chi_0$ has sufficiently deep conductor $c_{ns}:=\prod_{l\mid N_{ns}}l^{\tilde{\nu}(l)}$, for fixed $\tilde{\nu}(l)\geq \nu(l)$, at the non-split prime divisors of $N$. The role of the latter is optimal and two-fold. On one hand side, Hida's explicit Waldspurger formula requires depth of conductor of $\lambda$ at such primes, for otherwise the period integral vanishes. On the other side, sufficient ramification at these primes in the sense of Proposition 3.8 of \cite{JaLa} and Theorem 20.6 of \cite{Ja} contributes to +1 sign in the functional equation for the central critical L-value. Write $\lambda^-:=(\lambda \circ c)/|\lambda|$ for the unitary projection. We fix two embeddings $\iota_\infty :\bar{\mathbb{Q}} \hookrightarrow \mathbb{C}$ and  $\iota_p : \bar{\mathbb{Q}} \hookrightarrow \mathbb{C}_p$. Let $W$ denote the ring of Witt vectors with coefficients in an algebraic closure $\bar{\mathbb{F}}_p$ of the finite field of $p$ elements $\mathbb{F}_p$, regarded as a $p$-adically closed discrete valuation ring inside $p$-adic completion $\mathbb{C}_p$ of $\bar{\mathbb{Q}}_p$, and let $\mathfrak{P}$ be the prime of $W$ over $p$. Set $\mathcal{W}=\iota_p^{-1}(W)$ which is a strict henselization of $\mathbb{Z}_{(p)}=\mathbb{Q}\cap\mathbb{Z}_p$. Let $\mathrm{\Omega}_\infty\in\mathbb{C}^\times$ be the N\'eron period of a CM elliptic curve over $\mathcal{W}$.
We normalize the L-value as
\[
L^{\mathrm{alg}}(\frac{1}{2},\hat{\pi}_{\mathbf{f}}\otimes \lambda^-):= G\frac{\Gamma(k+m)\Gamma(m+1)}{\pi^{k+2m+1}}E(1/2)E'(m)\frac{L^{(N\ell d)}(\frac{1}{2},\hat{\pi}_{\mathbf{f}}\otimes \lambda^-)}{{\Omega_\infty^{2(k+2m)}}}\,,
\]
where we write $L^{(N\ell d)}(s,\hat{\pi}_{\mathbf{f}}\otimes \lambda^-)$ for the imprimitive L-function obtained by removing Euler factors at primes dividing $N\ell d$ from the primitive one, $E(1/2)$ and $E'(m)$ are modification Euler factors given by (\ref{euler1}) and (\ref{euler2}), respectively, and $G$ is essentially product of Gauss sums given by (\ref{gauss}). Let $\mathcal{S}(N,\ell)$ be a finite set of prime divisors of elements in
\[
 \{N, \ell-1\} \cup \{l-1\,:\text{ prime }l\mid N \text{ is ramified in }M\} \cup \{l-1,l+1\,:\text{ prime }l\mid N \text{ is inert in } M\}\,.
\]
Then our result states:
\begin{main}\label{main}
Let $p\not = \ell$ be two fixed rational primes such that $p$ splits in $M$. Let $f$ be a normalized Hecke newform of level $\Gamma_0(N)$, $N\geq 1$, weight $k\geq1$, and nebentypus $\psi$. Suppose $\ell \nmid N $ and that $p>2$ is outside the above finite set of primes $\mathcal{S}(N,\ell)$. Fix a Hecke character $\lambda$ as above, that is, of $\infty$-type $(k+m,-m)$ for arbitrary $m\geq 0$, whose conductor is supported at primes dividing $N\ell$ and is equal to a fixed preassigned $c_{ns}$ as above at ones that are non-split. Then
\[L^{\mathrm{alg}}(\frac{1}{2},\hat{\pi}_{\mathbf{f}}\otimes (\lambda\chi)^-)\not\equiv 0 \; (\mathrm{mod}\, \mathfrak{P})\]
for all but finitely many anticyclotomic characters $\chi$ of $\ell$-power conductor.
\end{main}
The above finite set of primes $p$ is excluded from consists of the primes which divide the fudge factors in Hida's computation of an explicit Waldspurger type of formula we use. The proof is actually valid under condition milder than $\ell\nmid N$ (see Assumption \ref{assumption}). We also prove validity of the Theorem \ref{main} when $\ell\|N$ and the local component $\pi_{\mathbf{f},\ell}$ is a special representation (see Proposition \ref{spcl}), but we hope to treat the general case when $\ell\mid N$ in a future paper.

The L-value in the Theorem \ref{main} is actually $L(\frac{1}{2},\hat{\pi}_{\mathbf{f}}\otimes (\lambda  \chi  |\cdot|_{M_\mathbb{A}}^m)^-)$, however the norm character has trivial unitary projection. The weight $k+2m+1$ of theta series $\theta(\lambda  \chi  |\cdot|_{M_\mathbb{A}}^m)$ associated to Hecke character $\lambda  \chi |\cdot|_{M_\mathbb{A}}^m$ being strictly greater than the weight $k$ of the cusp form $f$, incites essentially different arithmetic nature from the L-value studied in the landmark work of Vatsal started in deep and beautiful papers \cite{Va02} and \cite{Va03}, and continued in joint work with Cornut \cite{CoVa}, where the comparison of the weights is opposite. Thus, the period for the L-value there depends on $f$ only, namely, it is Hida's canonical period from \cite{H88ajm}, as opposed to a power of a CM period attached to $M$. The canonical Selmer group and the Main Conjecture associated to our L-values are distinct. Needless to say, as a direct application of Hida's method from \cite{HidaDwork}, our proof fits in Vatsal's philosophy about ergodic rigidity principle underlying non-vanishing of L-values, as explained in his ICM 2006 report \cite{VaICM}.

\subsection*{Acknowledgment} I would like to express my gratitude to Prof. Haruzo Hida for his generous insight, support and guidance. After his oral explanation of the circle of ideas, I merely had to work out details. I am also deeply indebted to Hae-Sang Sun for paving the path to this work by his dissertation \cite{HS}, in particular, using the \v Cebotarev density theorem and the Galois representation in the proof of the main theorem I learned from there.
\tableofcontents

\section{Modular forms}

\subsection{Classical modular forms and adelic ones}\label{sec:cmfadel}Let $G$ be algebraic group  $\mathrm{GL}(2)_{/\mathbb{Q}}$. 
We denote by $S_k(\Gamma_0(N),\psi)$ the space of holomorphic cusp forms $f$ of level $\Gamma_0(N)$ and nebentypus $\psi$ with $f(\gamma(z))=\psi(\gamma)f(z)j(\gamma,z)$ for $\gamma\in \Gamma_0(N)$, where $j(\bigl(\begin{smallmatrix} a & b \\ c & d \end{smallmatrix} \bigr),z)=cz+d$ for $z\in\mathfrak{H}=\{z|\mathrm{Im}(z)>0\}$ and $\bigl(\begin{smallmatrix} a & b \\ c & d \end{smallmatrix} \bigr)\in G(\mathbb{R})$. Here a Dirichlet character $\psi$ is regarded as a character of $\widehat{\Gamma}_0(N)$ via $\bigl(\begin{smallmatrix} a & b \\ c & d \end{smallmatrix} \bigr)\mapsto \psi(d)$. By virtue of strong approximation theorem $G(\mathbb{A})=G(\mathbb{Q})\widehat{\Gamma}_0(N)\mathrm{GL}_2^+(\mathbb{R})$ ($\mathrm{GL}_2^+(\mathbb{R})=\{ g \in G(\mathbb{R})|\mathrm{det}(g)>0\}$) allowing us to lift $f$ to $\mathbf{f}:G(\mathbb{Q})\backslash G(\mathbb{A})\to \mathbb{C}$ by $\mathbf{f}(\alpha u g_\infty)=f(g_\infty(\mathbf{i}))\psi(u)j(g_\infty,\mathbf{i})^{-k}$ for $\alpha\in G(\mathbb{Q})$, $u\in \widehat{\Gamma}_0(N)$ and $g_\infty \in \mathrm{GL}_2^+(\mathbb{R})$. Note that $\mathbf{f}(\alpha g u)=\psi(u)\mathbf{f}(g)$ for $\alpha\in G(\mathbb{Q})$ and $u\in \widehat{\Gamma}_0(N)$, and that $\mathbf{f}$ is so-called arithmetic lift of $f$ -- more widely used automorphic lift involves the determinant factor which we omitted. Denoting by $\mathcal{S}_k(\widehat{\Gamma}_0(N),\psi)$ the space of adelic cusp forms $\mathbf{f}$ obtained from $f\in S_k(\Gamma_0(N),\psi)$ in this way, we have $S_k(\Gamma_0(N),\psi)\cong\mathcal{S}_k(\widehat{\Gamma}_0(N),\psi)$ via $f\leftrightarrow\mathbf{f}$. Note that the center $Z(\mathbb{A})$ acts on $\mathcal{S}_k(\Gamma_0(N),\psi)$ via $\mathbf{f}|\zeta(g)=\mathbf{f}(\zeta g)$ and that $\mathbf{f}|\zeta_\infty=\zeta_\infty^{-k}\mathbf{f}$ for $\zeta\in Z(\mathbb{A})$. Thus $\mathcal{S}_k(\Gamma_0(N),\psi)$ decomposes into the direct sum of eigenspaces for this action and on each eigenspace $Z(\mathbb{A})$ acts by a Hecke character whose restriction to $\widehat{\Gamma}_0(N)\cap Z(\mathbb{A})$ is $\psi$ and which sends $\zeta_\infty$ to $\zeta_\infty^{-k}$. If we lift $\psi$ to $\mathbb{A}^\times$ in standard way and set $\boldsymbol{\psi}:=\psi |\cdot|_{\mathbb{A}}^{-k}$, let $\mathcal{S}_k(N,\boldsymbol{\psi})$ denote the $\boldsymbol{\psi}$-eigenspace. Then $S_k(\Gamma_0(N),\psi)\cong \mathcal{S}_k(N,\boldsymbol{\psi})$ via $f\leftrightarrow\mathbf{f}$.

\subsection{Algebro-geometric modular forms} Let $N$ be a positive integer, $B$ a fixed base $\mathbb{Z}[\frac{1}{N}]$-algebra and $S$ a $B$-scheme. An elliptic curve $E$ over $S$ is a proper smooth morphism $\pi : E\to S$ whose geometric fibers are connected curves of genus 1, together with a section $\textbf{0}:S\to E$. By level $\Gamma_1(N)$-structure, we refer to an embedding of finite flat group schemes $i_N:\mu_N\hookrightarrow E[N]$, where $E[N]$ is a scheme-theoretic kernel of multiplication by $N$ map -- it is a finite flat abelian group scheme over $S$ of rank $N^2$. 

The modular curve $\mathfrak{M}(\Gamma_1(N))$ of level $\Gamma_1(N)$ classifies pairs $(E,i_N)_{/S}$, for a $B$-scheme $S$. In other words, $\mathfrak{M}(\Gamma_1(N))$ is a coarse moduli scheme of the following functor from the category of $B$-schemes to the category \textit{SETS}
\[ \mathcal{P}(S)=[(E,i_N)_{/S}]_{/\cong}\]
where $[\quad]_{/\cong}$ denotes the set of isomorphism classes of the objects inside the brackets. When $N>3$ it is a fine moduli scheme of this functor.

Let $\omega$ denote a basis of $\pi_*(\Omega_{E/S})$, that is a nowhere vanishing section of $\Omega_{E/S}$. Fix a positive integer $k$ and a continuous character $\psi : (\mathbb{Z}/N\mathbb{Z})^\times\rightarrow B^\times$. A $B$-integral holomorphic modular form of weight $k$, level $\Gamma_0(N)$ and nebentypus $\psi$ is a function of isomorphism classes of $(E,i_N,\omega)_{/A}$, defined over $B$-algebra $A$, satisfying the following conditions:
\begin{itemize}
 \item[(G0)] $f((E,i_N,\omega)_{/A})\in A$ if $(E,i_N,\omega)$ is defined over $A$
 \item[(G1)] If $\varrho:A\rightarrow A'$ is a morphism of $B$-algebras then $f((E,i_N,\omega)_{/A}\otimes_{B} A')=\varrho(f((E,i_N,\omega)_{/A}))$;
 \item[(G2)] $f((E,i_N,a\omega)_{/A}))=a^{-k}f((E,i_N,\omega)_{/A})$ for $a\in A^\times=\mathbb{G}_m(A)$;
 \item[(G3)] $f((E,i_N\circ b,\omega)_{/A}))=\psi(b)f((E,i_N,\omega)_{/A})$ for $b\in (\mathbb{Z}/N\mathbb{Z})^\times$, where $b$ acts on $i_N$ by the canonical action of $\mathbb{Z}/N\mathbb{Z}$ on the finite flat group scheme $\mu_N$;
 \item[(G4)] For the Tate curve $Tate(q^N)$ over $B\otimes_{\mathbb{Z}}\mathbb{Z}((q))$ viewed as algebraization of the formal quotient        
   $\widehat{\mathbb{G}}_m/q^{N\mathbb{Z}}$, its canonical differential $\omega_{Tate}^{can}$ deduced from $\frac{\mathrm{d}t}{t}$  on $\widehat{\mathbb{G}}_m$, and all   
   $\Gamma_1(N)$-level structures $i_{Tate,N}$ coming from canonical images of points $\zeta_N^i q^j$ from  $\widehat{\mathbb{G}}_m$ ($0\leq i,j\leq N-1$), we have
   \[ f((Tate(q^N),i_{Tate,N},\omega_{Tate}^{can}))\in B\otimes_{\mathbb{Z}}\mathbb{Z}[[q]]\; .\] 
\end{itemize}
The space of $B$-integral holomorphic modular forms of weight $k$, level $\Gamma_0(N)$ and nebentypus $\psi$ is a $B$-module of finite type and we denote it by $G_k(N,\psi;B)$.

As it is well known, over $\mathbb{C}$, the category of test objects $(E,i_N,\omega)$ is equivalent to the category of the pairs $(L,i)$, where $L$ is a $\mathbb{Z}$-lattice in $\mathbb{C}$, and $i:\mathbb{Z}/NZ \hookrightarrow \frac{1}{N}L/L$. The differential $\omega$ can be recovered by pulling back $\mathrm{d}u$ to $E(\mathbb{C})=\mathbb{C}/L$, for standard variable $u$ on $\mathbb{C}$. Conversely,
\[L_E=\left\{ \int_{\gamma}\omega \in \mathbb{C} \,\big| \, \gamma \in H_1(E(\mathbb{C}),\mathbb{Z}) \right\} \]
is a lattice in $\mathbb{C}$. Then an algebro-geometric modular form $f$ integral over $\mathbb{C}$ gives rise to a classical modular form, whence also to adelic one, via
\[f(z):=f\left(L_z,i, 2\pi \mathbf{i} \mathrm{d}u \right) \,,\]
where $L_z=\mathbb{Z}+\mathbb{Z}z$ and $i:1 \mapsto 1/N$.

\subsection{$p$-adic modular forms}
\label{p-adicmf}
Fix a prime number $p$ that does not divide $N$. Let $B$ be  a $p$-adic algebra, that is, an algebra complete and separated in its $p$-adic topology. For an elliptic curve $E_{/S}$ we consider a Barsotti-Tate group $E[p^\infty]=\underrightarrow{\lim}_nE[p^n]$ for finite flat group schemes $E[p^n]$ equipped with closed immersions $E[p^n]\hookrightarrow E[p^m]$ for $m>n$ and the multiplication $[p^{m-n}]:E[p^m]\to E[p^n]$ which is an epimorphism in the category of finite flat group schemes. We consider a morphism of ind-group schemes $i_p:\mu_{p^\infty}\hookrightarrow E[p^\infty]$ which induces isomorphism of formal groups $\hat{i}_p:\widehat{\mathbb{G}}_m\cong\widehat{E}$ called trivialization of $E$. Here $\widehat{E}$ is the formal completion of $E$ along its zero-section.

A holomorphic $p$-adic modular form over $B$ is a function of isomorphism classes of $(E,i_N,i_p)_{/A}$, defined over $p$-adic $B$-algebra $A$, satisfying the following conditions:
\begin{itemize}
	\item[(P0)] $f((E,i_N,i_p)_{/A})\in A$ if $(E,i_N,i_p)$ is defined over $A$;
	\item[(P1)] If $\varrho:A\rightarrow A'$ is a $p$-adically continuous morphism of $B$-algebras then $f((E,i_N,i_p)_{/A}\otimes_{B} A')=\varrho(f((E,i_N,i_p)_{/A}))$;
	\item[(P2)]  For the Tate curve $Tate(q^N)$ over $\widehat{B((q))}$, which is a $p$-adic completion of $B((q))$, the canonical $p^\infty$-structure $i_{Tate,p}^{can}$ and canonical $\Gamma_0(N)$-level structure $i_{Tate,N}^{can}$, we have $f((Tate(q^N),i_{Tate,N}^{can},i_{Tate,p}^{can}))\in B[[q]]$. (This in fact implies that we have $f((Tate(q^N),i_{Tate,N},i_{Tate,p}))\in B[[q]]$ for all level structures $i_{Tate,N}$ and $i_{Tate,p}$ .)
\end{itemize}
We denote the space of $p$-adic holomorphic modular forms over $B$ by $V(N;B)$. 

The fundamental $q$-expansion principle holds for both algebro-geometric and $p$-adic modular forms (\cite{DeRa} Theorem VII.3.9 and \cite{Ka76} Section 5):
\begin{qexp}[$q$-expansion principle]
\begin{itemize}
	\item[1.] The $q$-expansion maps $G_k(N,\psi;B)\to B[[q]]$ and $V(N;B)\to B[[q]]$ are injective for any ($p$-adic) algebra $B$.
	\item[2.] Let $B\subset B'$ be ($p$-adic) algebras. The following commutative diagrams 
\begin{center} 
$\begin{CD}
G_k(N,\psi;B) @>>> B[[q]] \\
@VVV @VVV\\
G_k(N,\psi;B') @>>> B'[[q]]
\end{CD} \qquad \begin{CD}
V(N;B) @>>> B[[q]] \\
@VVV @VVV\\
V(N;B') @>>> B'[[q]]
\end{CD}$ 
\end{center}
\end{itemize}
are Cartesian, that is, the image of $G_k(N,\psi;B)$ in $G_k(N,\psi;B')$ ($V(N;B)$ in $V(N;B')$) is precisely the set of ($p$-adic) modular forms whose $q$-expansions have coefficients in B.
\end{qexp}
Using trivialization $\hat{i}_p$ we can push forward the canonical differential $\frac{\mathrm{d}t}{t}$ on $\widehat{\mathbb{G}}_m$ to obtain an invariant differential $\omega_p:=\hat{i}_{p,*}(\frac{\mathrm{d}t}{t})$ on $\widehat{E}$ which then extends to an invariant differential on $E$. Thus for $f\in G_k(N,\psi;B)$ we can define
\[f((E,i_N,i_p)):=f((E,i_N,\omega_p))\]
and we may regard an algebro-geometric holomorphic modular form as a $p$-adic one. By virtue of $q$-expansion principle, $G_k(N,\psi;B)\hookrightarrow V(N;B)$ is an injection preserving $q$-expansions.

\section{Rankin--Selberg L-functions and their special values}

Let $f$  be a normalized Hecke eigen-cusp form of level $\Gamma_0(N)$, $N\geq 1$, weight $k\geq1$, and nebentypus $\psi$ and let $\mathbf{f}$ be its corresponding adelic form on  $G(\mathbb{Q})\backslash G(\mathbb{A})$ with central character $\boldsymbol{\psi}$. The unitarization  $\mathbf{f}^u(g) := \mathbf{f}|\boldsymbol{\psi}(\mathrm{det}(g))|^{-1/2}$  of $\mathbf{f}$ generates a unitary automorphic representation  $\pi_\mathbf{f}$. Let $M$ be an imaginary quadratic field of discriminant $d(M)$. Let $\chi$ be an arithmetic Hecke character of $M^{\times} \backslash M^{\times}_{\mathbb{A}}$ and set $\chi^-:=(\chi \circ c)/|\chi|$ for its unitary projection. We denote by $L(s,\pi_\mathbf{f}\otimes \pi_{\chi^-})$ Rankin--Selberg L-function associated to $\pi_\mathbf{f}$ and automorphic unitary representation $\pi_{\chi^-}$ of $G$ attached to $\chi^-$ by theta series (see \cite{JaLa} and \cite{Ja} Section 19 for definitions). It is initially defined as a product of Euler factors over all places of $\mathbb{Q}$ and has a meromorphic continuation to $\mathbb{C}$ satisfying functional equation
\[L(s,\pi_\mathbf{f}\otimes \pi_{\chi^-})=\epsilon(s, \pi_\mathbf{f}\otimes \pi_{\chi^-})L(1-s,\tilde{\pi}_\mathbf{f}\otimes \pi_{(\chi^-)^{-1}})\]
where $\tilde{\pi}_\mathbf{f}$ denotes contragredient of $\pi_\mathbf{f}$ and $\epsilon(s, \pi_\mathbf{f}\otimes \pi_{\chi^-})$ is certain $\epsilon$-factor. Under key condition
\[\chi|_{\mathbb{A}^\times}= \boldsymbol{\psi}^{-1}\]
we know that $L(s,\pi_\mathbf{f}\otimes \pi_{\chi^-})$ is entire and equal to $L(1-s,\tilde{\pi}_\mathbf{f}\otimes \pi_{(\chi^-)^{-1}})$, whence the functional equation becomes
\[ L(s,\pi_\mathbf{f}\otimes \pi_{\chi^-})=\epsilon(s, \pi_\mathbf{f}\otimes \pi_{\chi^-}) L(1-s,\pi_\mathbf{f}\otimes \pi_{\chi^-})\,.\]
Thus the parity of order of vanishing of $L(s,\pi_\mathbf{f}\otimes \pi_{\chi^-})$ at central critical point $s=1/2$ is determined by the value of sign
\[\epsilon(\pi_\mathbf{f}\otimes \pi_{\chi^-}):=\epsilon(1/2, \pi_\mathbf{f}\otimes \pi_{\chi^-})\in \{\pm 1\}\,. \]
The order of vanishing is expected to be minimal for most characters $\chi^-$, in other words, either $L(1/2,\pi_\mathbf{f}\otimes \pi_{\chi^-})$ or $L'(1/2,\pi_\mathbf{f}\otimes \pi_{\chi^-})$ should be nonzero, depending whether the sign is $+1$ or $-1$, respectively.

The global sign $\epsilon(\pi_\mathbf{f}\otimes \pi_{\chi^-})$ is a product over all places $v$ of $\mathbb{Q}$ of local signs $\epsilon(\pi_{\mathbf{f},v}\otimes \pi_{\chi^-,v})$, which are attached to local components of $\pi_\mathbf{f}$ and $\pi_{\chi^-}$, normalized as in \cite{Gr}. If $\eta$ is quadratic Hecke character of $\mathbb{Q}$ attached to $M$ and we set
\[ S(\chi):= \{ v|\epsilon(\pi_{\mathbf{f},v}\otimes \pi_{\chi^-,v})\not = \eta_v\boldsymbol{\psi}_v(-1)\} \]
then the product formula $\eta\boldsymbol{\psi}(-1)=1=\prod_v \eta_v\boldsymbol{\psi}_v(-1)$ implies
\[ \epsilon(\pi_\mathbf{f}\otimes \pi_{\chi^-})=(-1)^{\#S(\chi)} \, . \]

In this paper we always work with Hecke characters $\chi$ of conductor not only supported at $N\ell$ but rather sufficiently deep there. Combined with the fact that the infinity type of $\chi$ is $(k+2m,0)$ so that $\chi_\infty(a)=a_\infty^{k+2m}$, this leaves the set $S(\chi)$ empty. Indeed, the local $\epsilon$-factors at places $v$ outside $N\ell d(M)$ are equal to 1, as both $\pi_{\mathbf{f},v}$ and $\pi_{\chi^-,v}$ are unramified principal series, and consequently these places do not belong to $S(\chi)$. At places $v|N\ell$, following Section 1.1 of \cite{CoVa}, we use a combination of Proposition 3.8 of \cite{JaLa} and Theorem 20.6 of \cite{Ja}, so that once we impose that $\chi$ is sufficiently ramified at these $v$, none of them belongs to $S(\chi)$. The only finite places that remain are such that $v|d(M)$ but $v\nmid N$. Note that here $\pi_{\mathbf{f},v}$ is unramified principal series and $\boldsymbol{\psi}_v$ is unramified, so in this situation we may use local calculation (3.1.1) and (3.1.2) in \cite{Zh} to conclude that these places do not belong to $S(\chi)$. Finally, as infinite order character $\chi$ has infinity type $(k+2m,0)$ we have $\epsilon(\pi_{\mathbf{f},\infty} \otimes \pi_{\chi^-,\infty})=(-1)^k$ essentially by Tate's thesis (\cite{Ta}) and the archimedean place does not belong to $S(\chi)$. In conclusion, the global $\epsilon$-factor is 1.

We may take the base-change lift  $\hat{\pi}_{\mathbf{f}}$ of $\pi_\mathbf{f}$ to $\mathrm{Res}_{M/\mathbb{Q}}G$ and consider Rankin--Selberg L-function $L(s,\hat{\pi}_{\mathbf{f}}\otimes \chi^-)$ due to the automorphic induction identity $L(s,\hat{\pi}_{\mathbf{f}}\otimes \chi^-)= L(s,\pi_\mathbf{f}\otimes \pi_{\chi^-})$.  

The rationality of central critical L-values $L(1/2,\hat{\pi}_{\mathbf{f}}\otimes \chi^-)$, after being divided by suitable periods, is proved by Shimura in \cite{Sh76}, and is consistent with the conjectures of Deligne. The nature of the period depends on the $\infty$-type of $\chi$, or more precisely, its comparison to the weight $k$ of $f$. In case $\chi$ is of $\infty$-type $(\kappa,0)$ and $\kappa\leq k-2$, the period depends on  $f$ but is independent of $\chi$, and can be expressed in terms of the Shimura periods $u_\pm(f)$. On the other hand, if $\kappa\geq k$, the period is independent of $f$ and depends on $M$ only -- it is a power of CM period attached to $M$.

\section{Shimura curves}
\subsection{Elliptic curves with complex multiplication}
It is well known that a $\mathbb{Z}$-lattice in $M$ is actually a proper ideal of a $\mathbb{Z}$-order $R(\mathfrak{a})=\{\alpha\in R|\alpha\mathfrak{a}\subset\mathfrak{a}\}$  of $M$. On the other hand, every $\mathbb{Z}$-order $\mathcal{O}$ of $M$ is of the form $\mathcal{O}=\mathbb{Z}+cR$ for a rational integer $c$ called the conductor and the following are equivalent (see Proposition 4.11 and (5.4.2) in \cite{IAT} and Theorem 11.3 of \cite{CRT})
\begin{itemize}
\item[(1)] $\mathfrak{a}$ is $\mathcal{O}$-projective fractional ideal
\item[(2)] $\mathfrak{a}$ is locally principal, i.e. \frenchspacing the localization at each prime is principal
\item[(3)] $\mathfrak{a}$ is a proper $\mathcal{O}$-ideal, i.e. \frenchspacing $\mathcal{O}=R(\mathfrak{a})$.  
\end{itemize}
Thus, one can define class group $\mathrm{Cl}_M^-(\mathcal{O}):=\mathrm{Pic}(\mathcal{O})$ to be the group of $\mathcal{O}$-projective fractional ideals modulo globally principal ideals. It is a finite group called the ring class group of conductor $c$, if $c$ denotes the conductor of $\mathcal{O}$. 

In this paper, we are concerned with orders $R_{c\ell^n}:=\mathbb{Z}+c\ell^nR$ and their ring class groups $\mathrm{Cl}^-_n:=\mathrm{Pic}(R_{c\ell^n})$ when $n\geq 0$, where $c$ is a fixed choice of integer prime to $\ell$ that will always be clear from the context. By class field theory, $\mathrm{Cl}^-_n$ is the Galois group $\mathrm{Gal}(H_{c\ell^n}/M)$ of the ring class field $H_{c\ell^n}$ of conductor $c\ell^n$. We define anticyclotomic class group modulo $c\ell^\infty$, $\mathrm{Cl}^-_\infty:=\underleftarrow{\lim}_n\mathrm{Cl}^-_n$ for the projection $\pi_{m+n,n}:\mathrm{Cl}^-_{m+n}\to\mathrm{Cl}^-_n$ taking $\mathfrak{a}$ to $\mathfrak{a}R_{c\ell^n}$. Then the group $\mathrm{Cl}^-_\infty$ is isomorphic to the Galois group of the maximal ring class field $H_{c\ell^\infty}=\bigcup_n H_{c\ell^n}$ of conductor $c\ell^\infty$ of $M$. 
If a $\mathbb{Z}$-lattice $\mathfrak{a}\subset M$ has $p$-adic completion $\mathfrak{a}_p=\mathfrak{a}\otimes_\mathbb{Z}\mathbb{Z}_p$ identical to $R\otimes_\mathbb{Z}\mathbb{Z}_p$, we consider a complex torus $\X{\mathfrak{a}}(\mathbb{C})=\mathbb{C}/\mathfrak{a}$. By the main theorem of complex multiplication (\cite{ACM} 18.6), this complex torus is algebraizable to an elliptic curve having complex multiplication by $M$ and defined over a number field. Then applying Serre--Tate's criterion of good reduction (\cite{SeTa}) we can conclude that \X{\mathfrak{a}} is actually defined over the field of fractions $\mathcal{K}$ of $\mathcal{W}$ and extends to an elliptic curve over $\mathcal{W}$ still denoted by \Xw{\mathfrak{a}}{\mathcal{W}}. All endomorphisms of \Xw{\mathfrak{a}}{\mathcal{W}} are defined over $\mathcal{W}$ and its special fiber $\Xw{\mathfrak{a}}{\bar{\mathbb{F}}_p}=\Xw{\mathfrak{a}}{\mathcal{W}}\otimes \bar{\mathbb{F}}_p$ is ordinary by assumption that $p=\mathfrak{p}\bar{\mathfrak{p}}$ splits in $M$.

Let $\mathcal{T}(\X{\mathfrak{a}})=\underleftarrow{\lim}_N\X{\mathfrak{a}}[N](\bar{\mathbb{Q}})$ be its Tate module. Choice of a $\widehat{\mathbb{Z}}$-basis $(w_1,w_2)$ of $\widehat{\mathfrak{a}}=\mathfrak{a}\otimes_{\mathbb{Z}}\widehat{\mathbb{Z}}$ gives rise to a level $N$-structure $\eta_N(\mathfrak{a}):(\mathbb{Z}/N\mathbb{Z})^2\cong \X{\mathfrak{a}}[N]$ given by $\eta_N(\mathfrak{a})(x,y)=\frac{xw_1+xw_2}{N}\in \X{\mathfrak{a}}[N]$. After taking their inverse limit and tensoring with $\mathbb{A}^{(\infty)}$ we get level structure
\[ \eta(\mathfrak{a}) = \underleftarrow{\lim}_N\eta_N(\mathfrak{a}):(\mathbb{A}^{(\infty)})^2\cong \mathcal{T}(\X{\mathfrak{a}})\otimes_{\widehat{\mathbb{Z}}}\mathbb{A}^{(\infty)}=:V(\X{\mathfrak{a}}) \; . \] 
We can remove $p$-part of $\eta(\mathfrak{a})$ and define level structure $\eta^{(p)}(\mathfrak{a})$ that conveys information about all prime-to-$p$ torsion in \X{\mathfrak{a}}:
\[ \eta^{(p)}(\mathfrak{a}):(\mathbb{A}^{(p\infty)})^2\cong
\mathcal{T}(\X{\mathfrak{a}})\otimes_{\widehat{\mathbb{Z}}}\mathbb{A}^{(p\infty)}=:V^{(p)}(\X{\mathfrak{a}}) \; . \]
Prime-to-$p$ torsion in \Xw{\mathfrak{a}}{\mathcal{W}} is unramified at $p$ and $\X{\mathfrak{a}}[N]$ for $p\nmid N$ is \'etale whence constant over $\mathcal{W}$, so the level structure $\eta^{(p)}(\mathfrak{a})$ is still defined over $\mathcal{W}$ (\cite{ACM} 21.1 and \cite{SeTa}). 

Since \Xw{\mathfrak{a}}{\mathcal{W}} has ordinary reduction over $\mathcal{W}$, we get ordinary part of level structure at $p$, $\ord{\mathfrak{a}}:\mu_{p^\infty}\hookrightarrow\X{\mathfrak{a}}[p^\infty]$, identifying $\mu_{p^\infty}$ with connected component $\X{\mathfrak{a}}[p^\infty]^\circ \cong \X{\mathfrak{a}}[\mathfrak{p}^\infty]$. The Cartier duality then yields \'etale part of level structure at $p$, $\et{\mathfrak{a}}:\mathbb{Q}_p/\mathbb{Z}_p\cong\X{\mathfrak{a}}[p^\infty]^{\acute{e}t}\cong \X{\mathfrak{a}}[\bar{\mathfrak{p}}^\infty]$ over $\mathcal{W}$.
In this way we constructed a triple
\[\xw{\mathfrak{a}}{\mathcal{W}}\]
to which we refer as a test object.

\subsection{Definitions and basic facts}\label{sec:dbf}
The pairs $(X,\eta^{(p)})_{/S}$ for a $\mathbb{Z}_{(p)}$-scheme $S$, consisting of elliptic curve $X$ over $S$ and a $\mathbb{Z}$-linear isomorphism $\eta^{(p)}:(\mathbb{A}^{(p\infty)})^2\cong \mathcal{T}^{(p)}(X)\otimes_{\widehat{\mathbb{Z}}}\mathbb{A}^{(\infty)}=:V(X)$, are classified up to isogenies of degree prime to $p$ by a $p$-integral model $Sh_{/\mathbb{Z}_{(p)}}^{(p)}$ (\cite{Ko}) of the Shimura curve  $Sh_{/\mathbb{Q}}$ associated to algebraic group $G=\mathrm{GL}(2)_{/\mathbb{Q}}$. $Sh$ was initially constructed by Shimura in \cite{Sh66} but nicely reinterpreted by Deligne in \cite{De71} 4.16-4.22 as a moduli of abelian schemes up to isogenies. By its construction, $Sh^{(p)}$ is smooth $\mathbb{Z}_{(p)}$-scheme representing moduli functor $\mathcal{F}^{(p)}$ from the category of  $\mathbb{Z}_{(p)}$-schemes to \textit{SETS}:
\[ \mathcal{F}^{(p)}(S)=\{(X,\eta^{(p)})_{/S}\}_{/\approx} \] 
and two pairs $(X,\eta^{(p)})_{/S}$ and $(X',\eta'^{(p)})_{/S}$ are isomorphic up to prime-to-$p$ isogeny, which we write $(X,\eta^{(p)})_{/S}\approx (X',\eta'^{(p)})_{/S}$, if there exists an isogeny $\phi:X_{/S}\to X'_{/S}$ of degree prime to $p$ such that $\phi\circ\eta=\eta'$.

The pair $x(\mathfrak{a})= \xunrw{\mathfrak{a}}{\mathcal{W}}$ for a $\mathbb{Z}$-lattice $\mathfrak{a}$ with $\mathfrak{a}\otimes_\mathbb{Z}\mathbb{Z}_p =R\otimes_\mathbb{Z}\mathbb{Z}_p$, constructed as above, gives rise to a $\mathcal{W}$-point on $Sh^{(p)}$ to which we refer as CM point.

Each adele $g\in G(\mathbb{A}^{(\infty)})$ acts on a level structure $\eta$ by $\eta\mapsto\eta\circ g^{(p\infty)}$ inducing $G(\mathbb{A}^{(\infty)})$-action on $Sh$.
A sheaf theoretic coset $\bar{\eta}=\eta K$ for an open compact subgroup $K\subset G(\mathbb{A}^{(\infty)})$ maximal at $p$ (i.e. $K_p=G(\mathbb{Z}_p)$) is called a level $K$-structure. The quotient $Sh_K=Sh/K$ represents the following quotient functor
\[\mathcal{F}_K(S)=\{(E,\bar{\eta})_{/S}\}_{/\approx} \qquad \] 
and $Sh=\underleftarrow{\lim}_K Sh_K$ when $K$ runs over open compact subgroups of $G(\mathbb{A}^{(p\infty)})$. Let $V_{K/\bar{\mathbb{Q}}}$ denote geometrically irreducible component of $Sh^{(p)}_K\otimes_\mathbb{Q}\bar{\mathbb{Q}}$ containing geometric point $x(\mathfrak{a})=x(\mathfrak{a})\otimes_\mathbb{Q}\bar{\mathbb{Q}}$. The field of definition of $V_K$ in the sense of Weil is contained in $\mathcal{K}$, and we can think of schematic closure $V_{K/\mathcal{W}}$ in $Sh^{(p)}_{K/\mathcal{W}}$ that is smooth over $\mathcal{W}$ if $K^{(p)}$ is sufficiently small. The special fiber $V_K\otimes_{\mathcal{W}}\bar{\mathbb{F}}_p$ remains irreducible (see \cite{HidaDwork} Section 2.2 for details). We define
\[ V^{(p)}:=\underleftarrow{\lim}_K V_{K/\mathcal{W}}\]
where $K$ ranges over all open compact subgroups of $G(\mathbb{A}^{(\infty)})$ maximal at $p$. In conclusion, the scheme $V^{(p)}$ is smooth over $\mathcal{W}$, and its generic and special fibers are geometrically irreducible. In this sense, we refer to $V^{(p)}_{/\mathcal{W}}$ as geometrically connected component of $Sh^{(p)}$ containing the $\mathcal{W}$-point $x(\mathfrak{a})$.

When $K$ is chosen suitably (\cite{PAF} Section 4.2.1), $V_K$ is isomorphic to a, fine or coarse, moduli scheme $\mathfrak{M}(\Gamma(N))_{/\mathbb{Z}_{(p)}[\mu_N]}$ of level $\Gamma(N)$ (the principal congruence subgroup) representing functor
from the category of $\mathbb{Z}_{(p)}$-schemes to the category \textit{SETS}
\[ \mathcal{P}_{\Gamma(N)}(S)=[(E,\phi_N:(\mathbb{Z}/N\mathbb{Z})\cong E[N])_{/S}]_{/\cong}\]
and consequently
\[V^{(p)}_{/\mathcal{W}}\cong\underleftarrow{\lim}_{p\nmid N} \mathfrak{M}(\Gamma(N))_{/\mathcal{W}} \,. \]
Clearly,
\[ \mathfrak{M}(\Gamma(N))_{/\mathbb{Z}[1/N,\mu_N]}=\bigsqcup_{\zeta \in \mu_N} \mathfrak{M}(\Gamma(N),\zeta)\] 
where $\mathfrak{M}(\Gamma(N),\zeta)$ is a modular curve of level $\Gamma(N)$ representing functor from the category of $\mathbb{Z}[1/N]$-schemes to the category \textit{SETS}
\[ \mathcal{P}_{\Gamma(N), \zeta}(S)=[(E,\phi_N:(\mathbb{Z}/N\mathbb{Z})\cong E[N])_{/S}|\langle \phi_N(1,0),\phi_N(0,1)\rangle=\zeta]_{/\cong}\]
and $\langle \cdot,\cdot\rangle$ denotes the Weil pairing.

The complex points of the Shimura curve $Sh$ have the following expression:
\[ Sh(\mathbb{C})= G(\mathbb{Q})\left\backslash \left( \mathfrak{X}\times G(\mathbb{A}^{(\infty)})\right) \right/ Z(\mathbb{Q})\]
where $Z$ stands for center of $G$ and the action is given by $\gamma(z,g)u=(\gamma(z),\gamma gu)$ for $\gamma \in G(\mathbb{Q})$ and $u\in Z(\mathbb{Q})$ (\cite{De79} Proposition 2.1.10 and \cite{Mi} page 324 and Lemma 10.1). 

To each point $(X,\eta)\in Sh$ we can associate a lattice $\widehat{L}=\eta^{-1}(\mathcal{T}(X))\subset (\mathbb{A}^{(\infty)})^2$ and the level structure $\eta$ is determined by the choice of the basis $w=(w_1,w_2)$ of $\widehat{L}$ over $\widehat{\mathbb{Z}}$. In the view of basis $w$, the $G(\mathbb{A}^{(\infty)})$-action on $Sh$ given by $(X,\eta)\mapsto(X,\eta\circ g)$ is a matrix multiplication $w^\intercal\mapsto g^{-1}w^\intercal$ because $(\eta\circ g)^{-1}(\mathcal{T}(X))=g^{-1}\eta^{-1}(\mathcal{T}(X))=g^{-1}\widehat{L}$, where $\intercal$ stands for a transpose. We warn the reader about the following
\begin{switch} \label{switch} { \em
Insisting on modular point of view, the action of matrix $g^{-1}$ records change of the basis vectors themselves, rather than coordinates with respect to the basis. Having this on mind and desiring to view modular forms in adelic, algebro-geometric and $p$-adic phrasing in coherent way, it becomes more convenient for us to use identifications $R\otimes_{\mathbb{Z}}\mathbb{Z}_p=R_{\bar{\mathfrak{p}}}\oplus  R_{\mathfrak{p}}$ and  $\X{R}[p^n]=\X{R}[\bar{\mathfrak{p}}^n]\oplus\X{R}[\mathfrak{p}^n]=\mathbb{Z}/p^n\mathbb{Z}\oplus\mu_{p^n}$, $n\geq 1$, in constructing level structures for our CM points due to the definition of nebentypus. }
\end{switch}

\subsection{Hecke relation among CM points}\label{sec:heckecmpts}
Let $c$ be a fixed positive integer prime to $\ell$, and $n\geq 0$. In this section we associate to each proper $R_{c\ell^n}$-ideal $\mathfrak{a}$ prime to $p$ a CM point on $Sh^{(p)}$ and describe Hecke relation among these points. 

A choice of $\widehat{\mathbb{Z}}$-basis $(w_1,w_2)$ of $\widehat{R}=R\otimes_{\mathbb{Z}}\widehat{\mathbb{Z}}$ gives rise to a level structure $\eta^{(p)}(R):(\mathbb{A}^{(p\infty)})^2\cong V^{(p)}(\X{R})$ defined over $\mathcal{W}$ as explained above. 
We fix such a basis so that its $p$-component $(w_{1,p},w_{2,p})$ is given by $(e_{\bar{\mathfrak{p}}}, e_{\mathfrak{p}})$ where $e_{\bar{\mathfrak{p}}}$ and $e_{\mathfrak{p}}$ are idempotents of $R_{\bar{\mathfrak{p}}}$ and $R_{\mathfrak{p}}$, respectively. At $\ell$, we can write $R\otimes_{\mathbb{Z}}\mathbb{Z}_\ell=\mathbb{Z}_\ell[\sqrt{d}]$ with $d\in \mathbb{Z}_\ell$, choosing $d\in  \mathbb{Z}_\ell^\times$ if $R\otimes_{\mathbb{Z}}\mathbb{Z}_\ell$ is unramified over $\mathbb{Z}_\ell$, and set $(w_{1,\ell},w_{2,\ell})=(\sqrt{d},1)$. We also fix once and for all an invariant differential $\omega(R)$ on \Xw{R}{\mathcal{W}} so that $H^0(\X{R},\Omega_{\Xw{R}{\mathcal{W}}})=\mathcal{W}\omega(R)$. 

Let $\mathfrak{a}$ be a proper $R_{c\ell^n}$-ideal whose $p$-adic completion $\mathfrak{a}_p=\mathfrak{a}\otimes_\mathbb{Z}\mathbb{Z}_p$ is identical to $R\otimes_\mathbb{Z}\mathbb{Z}_p$. Regarding $c\ell^n$ as an element of $\mathbb{A}^\times$, $(c\ell^nw_1,w_2)$ is a basis of $\widehat{R}_{c\ell^n}$ over $\widehat{\mathbb{Z}}$ yielding a level structure $\eta^{(p)}(R_{c\ell^n}):(\mathbb{A}^{(p\infty)})^2\cong V^{(p)}(X(R_{c\ell^n}))$. Choosing a complete set of representatives $\{a_1,\ldots,a_{H^-}\}\subset M^{\times}_{\mathbb{A}}$ so that $M^{\times}_{\mathbb{A}}=\bigsqcup_{j=1}^{H^-}M^\times a_j\widehat{R}_{c\ell^n}^\times M_\infty^\times$ we have $\widehat{\mathfrak{a}}=\alpha a_j\widehat{R}_{c\ell^n}$ for some $\alpha\in M^\times$ and $1\leq j\leq H^-$ and we can define $\eta^{(p)}(\mathfrak{a})=\alpha^{-1} a_j^{-1} \eta^{(p)}(R_{c\ell^n})$ so that we have commutative diagram
\[\begin{CD}
\mathcal{T}^{(p)}(\X{\mathfrak{a}}) @<\alpha a_j<< \mathcal{T}^{(p)}(\X{R_{c\ell^n}})\\
@A{\cong}A\eta^{(p)}(\mathfrak{a})A @A{\cong}A\eta^{(p)}(R_{c\ell^n})A\\
\widehat{\mathfrak{a}}^{(p)} @<\alpha a_j<< \widehat{R}_{c\ell^n}^{(p)}
\end{CD}\qquad .\]
Note that $\omega(R)$ induces a differential $\omega(\mathfrak{a})$ on \X{\mathfrak{a}} first by pulling back $\omega(R)$ from  \X{R} to \X{R\cap\mathfrak{a}} and then by pull-back inverse from \X{R\cap\mathfrak{a}} to \X{\mathfrak{a}}.

Assume from now that $n\geq 1$.
Let $C\subset \X{R_{c\ell^n}}[\ell]$ be a rank $\ell$ subgroup scheme of a finite flat group scheme $\X{R_{c\ell^n}}[\ell]$ that is \'etale locally isomorphic to $\mathbb{Z}/\ell\mathbb{Z}$ after faithfully flat extension of scalars. We consider geometric quotient of \X{R_{c\ell^n}} by such finite flat subgroup schemes $C$ (\cite{ABV} Section 12). The quotient map $\pi:\X{R_{c\ell^n}}\twoheadrightarrow\X{R_{c\ell^n}}/C$ is \'etale over $\mathcal{W}$ so we transfer level structure $\pi_{*}\eta^{(p)}(R_{c\ell^n})=\pi\circ\eta^{(p)}(R_{c\ell^n})$, as well as an invariant differential $(\pi^*)^{-1}\omega(R_{c\ell^n})$ to $\X{R_{c\ell^n}}/C$. 

Note that ideal $\ell_n=\ell\mathbb{Z} + \ell^n R= \ell R_{c\ell^{n-1}}$ is a prime ideal of $R_{c\ell^n}$ but not a proper one -- it is a proper 
ideal of $R_{c\ell^{n-1}}$, and we have $X(R_{c\ell^n})[\ell_n]= R_{c\ell^{n-1}} / R_{c\ell^n}$, so $X(R_{c\ell^n})/X(R_{c\ell^n})[\ell_n]\cong
X(R_{c\ell^{n-1}})$. Assume now that $C\subset X(R_{c\ell^n})[\ell]$ is different from $X(R_{c\ell^n})[\ell_n]$. Note that there are precisely $\ell$ such group subschemes. If $\mathfrak{a}$ is a lattice so that $\X{R_{c\ell^n}}/C=\X{\mathfrak{a}}$ then $\mathfrak{a}/R_{c\ell^n}=C$ and $\mathfrak{a}$ is a $\mathbb{Z}$-lattice of $M$ because $C$ is $\mathbb{Z}$-submodule. Since $\ell C=0$ we have $\ell R_{c\ell^n}\mathfrak{a}\subset\mathfrak{a}$ which means that $\mathfrak{a}$ is $R_{c\ell^{n+1}}$-ideal. Moreover $\mathfrak{a}$ is not $R_{c\ell^n}$-submodule so we conclude that $\mathfrak{a}$ is a proper $R_{c\ell^{n+1}}$-ideal. Since $C$ generates over $R_{c\ell^n}$ all $\ell$-torsion points of $X(R_{c\ell^n})$, we have $\mathfrak{a}R_{c\ell^n}=\ell^{-1}R_{c\ell^n}$, so proper ideal class of $\mathfrak{a}$ in $\mathrm{Cl}^-_{n+1}$ projects down to (identity) class of $R_{c\ell^n}$ in $\mathrm{Cl}^-_n$. In conclusion, this procedure gives rise to $\ell$ CM points $x(\mathfrak{a})$ on $Sh$ for representatives $\mathfrak{a}$ of precisely $\ell$ proper ideal classes in $\mathrm{Cl}^-_{n+1}$ that project down to a given (identity) class in $\mathrm{Cl}^-_n$.

It was computed in Section 3.1 of \cite{HidaDwork} that, in the sense of Deligne's interpretation of $Sh$, corresponding $G(\mathbb{A}^{(\infty)})$-action is given by $x(\mathfrak{a})=x(R_{c\ell^n})\circ \bigl(\begin{smallmatrix} 1 & \frac{u}{\ell} \\ 0 & 1 \end{smallmatrix} \bigr)$ when $u$ ranges through $\mathbb{Z}/\ell\mathbb{Z}$ (see (3.1) in Section 3.1 of \cite{HidaDwork}). Thus, for any fixed $\mathfrak{b}_0\in \mathrm{Cl}^-_{n}$ and $\mathfrak{a}_0\in \mathrm{Cl}^-_{n+1}$, such that $ \pi_{n+1,n}(\mathfrak{a}_0)=\mathfrak{b}_0$, we have
\[
\{x(\mathfrak{a})|\mathfrak{a}\in \mathrm{Cl}^-_{n+1} \text{ and } \pi_{n+1,n}(\mathfrak{a})=\mathfrak{b}_0\}=\{x(\mathfrak{a}_0)\circ \bigl(\begin{smallmatrix} 1 & \frac{u}{\ell} \\ 0 & 1 \end{smallmatrix} \bigr) | u\in \mathbb{Z}/\ell\mathbb{Z} \}\,.
\]
More generally, for any $s\geq 1$, and fixed $\mathfrak{b}_0\in \mathrm{Cl}^-_{n}$, $\mathfrak{a}_0\in \mathrm{Cl}^-_{n+s}$ such that $ \pi_{n+s,n}(\mathfrak{a}_0)=\mathfrak{b}_0$, we have
\begin{equation}\label{proj}
\{x(\mathfrak{a})|\mathfrak{a}\in \mathrm{Cl}^-_{n+s} \text{ and } \pi_{n+s,n}(\mathfrak{a})=\mathfrak{b}_0\}=\{x(\mathfrak{a}_0)\circ \bigl(\begin{smallmatrix} 1 & \frac{u}{\ell^s} \\ 0 & 1 \end{smallmatrix} \bigr) | u\in \mathbb{Z}/\ell^s\mathbb{Z} \}\,.
\end{equation}

\subsection{Isogeny action on modular forms} \label{sec:isog}
Let $N$ be a positive integer prime to $p$. Note that a point $x=(E,\eta^{(p)})\in Sh^{(p)}(S)$, for a $\mathcal{W}$-scheme $S$, projects down to $x=(E,i_N)\in \mathfrak{M}(\Gamma_1(N))(S)$ where $i_N= \eta^{(p)} \; \mathrm{mod} \; \widehat{\Gamma}_1(N)$ and the coset is taken as a sheaf theoretic one. Thus, if $\underline{X}=(X,\eta^{(p)},\eta_p^{\mathrm{ord}}, \omega)$ is a test object, an algebro-geometric modular form $f\in G_k(\Gamma_0(N),\psi;\mathcal{W})$ can be evaluated at this test object via
\[f(\underline{X}):=f(X, i_N, \omega )\]
and the same holds for a $p$-adic modular form $f\in V(N;W)$ via
\[ f(\underline{X}):= f(X, i_N, \eta_p^{\mathrm{ord}}) \,.\]
In this paper, $\widehat{\Gamma}_0(N)=\{\bigl(\begin{smallmatrix} a & b \\ c & d \end{smallmatrix} \bigr)\in G(\widehat{\mathbb{Z}})|\, c\in N\widehat{\mathbb{Z}} \}$ and $\widehat{\Gamma}_1(N)=\{\bigl(\begin{smallmatrix} a & b \\ c & d \end{smallmatrix} \bigr)\in \widehat{\Gamma}_0(N) |\, d-1\in N\widehat{\mathbb{Z}} \}$.
 
Second, by universality, we get a morphism $\mathfrak{M}(N)\to Sh^{(p)}/\widehat{\Gamma}_1(N)$ and the image of $\mathfrak{M}(N)$ gives a geometrically irreducible component of $Sh^{(p)}/\widehat{\Gamma}_1(N)$. The $G(\mathbb{A}^{(\infty)})$-action on $Sh^{(p)}$ given by $\eta^{(p)}\mapsto \eta^{(p)}\circ g^{(p)}$ is geometric preserving the base scheme $\mathop{Spec}(\mathcal{W})$. 

Let $q$ be a prime outside $N p\ell$. For any pair $(X, \eta^{(p)})_{/S}$, where $S$ is a $\mathcal{W}$-scheme, taking sheaf theoretic coset $\eta^{(p)}\;\mathrm{mod}\;\widehat{\Gamma}_0(q)$ induces a finite group subscheme $C$ of $X$ defined over $S$ and isomorphic to $\mathbb{Z}/q\mathbb{Z}$ \'etale locally. Thus, it makes sense to consider a level $\Gamma_0(q)$ test object $\underline{X}=(X,C,\eta^{(pq)},\eta_p^{\mathrm{ord}}, \omega)$ and we can construct canonically its image under $q$-isogeny
\[[q](X,C,\eta^{(pq)},\eta_p^{\mathrm{ord}}, \omega)=(X/C, \pi_*\eta^{(pq)}, \bar{\eta}_q, \pi_*\eta_p^{\mathrm{ord}},(\pi^*)^{-1}\omega )\]
for the projection $\pi:X\twoheadrightarrow X/C$, and $\bar{\eta}_q=\eta_q\cdot G(\mathbb{Z}_q)$ for any $\eta_q:\mathbb{Z}_q\cong\mathcal{T}_q(X/C)$.
This gives rise to a linear operator $[q]:V(N;W)\to V(qN;W)$ by $f|[q](\underline{X})=f([q](\underline{X}))$. For a test object $(L,C,i)$ in the lattice viewpoint, this is tantamount to choosing a lattice $L_C \supset L$ with $L_C/L \cong \mathbb{Z}/q\mathbb{Z}$ and $f|[q](L,C,i)=f(L_C,i)$. Over $\mathbb{C}$, from the classical point of view this is nothing but $f|[q](z)=\psi(q)q^kf(qz)$.

Let $(X(\mathfrak{a}),\eta^{(p)}(\mathfrak{a}), \ord{\mathfrak{a}}, \omega({\mathfrak{a}}))$ be a test object associated to $\mathbb{Z}$-lattice $\mathfrak{a}$ of conductor prime to $p$. If $q=\mathfrak{q}\bar{\mathfrak{q}}$ splits in $M$, then we can choose $\eta_q$ to be induced by \[X(\mathfrak{a})[q^\infty]\cong X(\mathfrak{a})[\bar{\mathfrak{q}}^\infty] \oplus  X(\mathfrak{a})[\mathfrak{q}^\infty] \cong M_{\bar{\mathfrak{q}}}/R_{\bar{\mathfrak{q}}} \oplus M_{\mathfrak{q}}/R_{\mathfrak{q}} \cong \mathbb{Q}_q/\mathbb{Z}_q \oplus \mu_{q^\infty}\]
so that we have a level $\Gamma_0(q)$-structure $C=X(\mathfrak{a})[\mathfrak{q}]$ on $X(\mathfrak{a})$ which depends on the choice of factor $\mathfrak{q}$. Then $[q](X(\mathfrak{a}))=X(\mathfrak{q}^{-1}\mathfrak{a})$. 

In summary, when $q$ is split in $M$, $q$-isogeny action $[q]$ on $Sh$ corresponds to $g\in G(\mathrm{A}^{(\infty)})$-action by (concentrated-at-$q$) matrix $\bigl(\begin{smallmatrix} 1 & 0 \\ 0 & q \end{smallmatrix} \bigr)\in G(\mathbb{Q}_q)\subset G(\mathbb{A}^{(\infty)})$ and we can write $f|[q]=f|\bigl(\begin{smallmatrix} 1 & 0 \\ 0 & q \end{smallmatrix} \bigr)$ where we define action 
\begin{equation}\label{isogmf}
f|g(X,\eta,\omega) =f(X, \eta\circ g, \omega)\,.
\end{equation} 
for $g\in G(\mathrm{A}^{(\infty)})$ and $f\in V(N;W)$.

\subsection{Differential operators}\label{diffop}

Recall the definition of Maass--Shimura differential operators on $\mathfrak{H}$ indexed by $k\in\mathbb{Z}$:
\[ \delta_k = \frac{1}{2\pi\mathbf{i}}\left(\frac{\partial}{\partial z}+\frac{k}{z-\bar{z}}\right)\; \text{and}\; \delta_k^r=\delta_{k+2r-2}\ldots\delta_k \]
for a non-negative integer $r$. They preserve rationality of a value at a CM point (\cite{AAF} III and \cite{Sh75}) when applied to modular forms. Let $\mathfrak{a}$ be a proper $R_{cp^n}$-ideal prime to $p$. 
Note that the complex uniformization $\X{\mathfrak{a}}(\mathbb{C})=\mathbb{C}/\mathfrak{a}$ induces a canonical invariant differential $\omega_\infty(\mathfrak{a})$ in $\Omega_{\X{\mathfrak{a}}/\mathbb{C}}$ by pulling back $\mathrm{d}u$, where $u$ is the standard variable on $\mathbb{C}$. Then one can define a period $\Omega_\infty\in\mathbb{C}^\times$ by $\omega(\mathfrak{a})=\Omega_\infty\omega_\infty(\mathfrak{a})$ (\cite{Ka} Lemma 5.1.45). Note that $\Omega_\infty$ does not depend on $\mathfrak{a}$ since $\omega(\mathfrak{a})$ is induced by $\omega(R)$ on \X{R} by construction. Then for $f\in G_k(N,\psi;\mathcal{W})$ we have
\[ \qquad \frac{\delta_k^rf(x(\mathfrak{a}),\omega_\infty(\mathfrak{a}))}{\Omega_\infty^{k+2r}}=\delta_k^rf(x(\mathfrak{a}),\omega(\mathfrak{a}))\in\mathcal{W}\]
(\cite{Ka} Theorem 2.4.5).

Katz introduced a purely algebro-geometric definition of Maass--Shimura differential operator (\cite{Ka} Chapter II) by interpreting it in terms of Gauss--Manin connection of the universal abelian variety with real multiplication over $\mathfrak{M}$. In this way he extended the operator $\delta_* $ to algebro-geometric and $p$-adic modular forms; we denote the latter extension of $\delta^r_*$ by $d^r:V(N;W)\to V(N;W)$. The ordinary part of level structure at $p$, $\ord{\mathfrak{a}}:\mu_{p^\infty}\cong \X{\mathfrak{a}}[\mathfrak{p}^\infty]$ induces trivialization $\widehat{\mathbb{G}}_m\cong \widehat{\X{\mathfrak{a}}}$ for the $p$-adic formal completion $\widehat{\X{\mathfrak{a}}}_{/W}$ of \X{\mathfrak{a}} along its zero-section. We obtain an invariant differential $\omega_p(\mathfrak{a})$ on $\widehat{\X{\mathfrak{a}}}_{/W}$ by pushing forward $\frac{\mathrm{d}t}{t}$ on $\widehat{\mathbb{G}}_m$, which then extends to an invariant differential on \Xw{\mathfrak{a}}{W} also denoted by $\omega_p(\mathfrak{a})$. Then one can define a period $\Omega_p\in W^\times$, independent of $\mathfrak{a}$, by $\omega(\mathfrak{a})=\Omega_p \omega_p(\mathfrak{a})$ (\cite{Ka} Lemma 5.1.47). The fact that will be of instrumental use for us is
\begin{equation}\label{K-S} \frac{(d^rf)(x(\mathfrak{a}),\omega_p(\mathfrak{a}))}{\Omega_p^{k+2r}}=(d^rf)(x(\mathfrak{a}),\omega(\mathfrak{a}))=(\delta_k^rf)(x(\mathfrak{a}),\omega(\mathfrak{a}))\in\mathcal{W}
\end{equation}
(\cite{Ka} Theorem 2.6.7). The effect of $d^m$ on $q$-expansion of a modular form is given by
\begin{equation}\label{diffqexp}
d^m \sum_{n\geq 0}a(n,f)q^n = \sum_{n\geq0}n^ma(n,f)q^n\,.
\end{equation}
(\cite{Ka} (2.6.27)).

\section{Hida's explicit Waldspurger formula}
\label{sec:hewf}

We adelize Maass--Shimura $m$-th derivative $\delta_k^mf$, $m\geq 0$, to a function $\mathbf{f}_m$ on 
$G(\mathbb{A})$ as in Section 3.1 of \cite{HidaNV}. We regard $X=\frac{1}{2}\bigl(\begin{smallmatrix} 1 & \mathbf{i} \\ \mathbf{i} & -1 \end{smallmatrix} \bigr)\in\mathfrak{sl}_2(\mathbb{C})$ -- a Lie algebra of $\mathrm{SL}_2(\mathbb{C})$, as an invariant differential operator $X_{g_\infty}$ on $\mathrm{SL}_2(\mathbb{C})$ for the variable matrix $g_\infty\in G(\mathbb{R})$ (here identifying $G(\mathbb{R})$ with $\mathrm{SL}_2(\mathbb{R})\times\mathbb{R}^\times$ by the natural isogeny), and set
\[ \mathbf{f}_m(g)=(-4\pi)^{-m}|\mathrm{det}(g)|_\mathbb{A}^{-m}X^m_{g_\infty}\mathbf{f}(g) \]
where $g_\infty$ is infinite part of $g\in G(\mathbb{A})$. Then $\mathbf{f}_m(g_\infty)=(\delta_k^mf)(g_\infty(\mathbf{i}))j(g_\infty,\mathbf{i})^{-k-2m}$, and 
when $\mathrm{det}(g_{\infty})=1$ we have
\begin{equation}\label{der}
\mathbf{f}_m(g)=|\mathrm{det}(g^{(\infty)})|_\mathbb{A}^{-m} \delta_k^m\mathbf{f}(g)
\end{equation}
where $\delta_k^m\mathbf{f}:G(\mathbb{Q})\backslash G(\mathbb{A})\to \mathbb{C}$ is the arithmetic lift of $\delta_k^mf$ as above, given by $\delta_k^m\mathbf{f}(\alpha u g_\infty)=\delta_k^mf(g_\infty(\mathbf{i}))\psi(u)j(g_\infty,\mathbf{i})^{-k-2m}$ for $\alpha\in G(\mathbb{Q})$, $u\in \widehat{\Gamma}_0(N)$ and $g_\infty \in \mathrm{GL}_2^+(\mathbb{R})$ (\cite{HidaNV} Definition 3.3 and Lemma 3.1). Here $g^{(\infty)}$ is finite part of $g\in G(\mathbb{A})$. The central character of $\mathbf{f}_m$ is given by $\boldsymbol{\psi}_m(x)=\boldsymbol{\psi}(x)|x|_{\mathbb{A}}^{-2m}$ and $\mathbf{f}_m(gu)=\boldsymbol{\psi}_m(u)\mathbf{f}_m(g)$ when $u\in \widehat{\Gamma}_0(N)$.

Let $f_0\in S_k(\Gamma_0(N),\psi)$ be a normalized Hecke newform of conductor $N$, nebentypus $\psi$ and let  $\mathbf{f}_0\in\mathcal{S}_k(N,\boldsymbol{\psi})$ be the corresponding adelic form with central character $\boldsymbol{\psi}$. Let $f$ be a suitable normalized Hecke eigen-cusp form that will be explicitly made out of $f_0$ such that its arithmetic lift $\mathbf{f}$ is inside the automorphic representation $\pi_{\mathbf{f}_0}$ generated by the unitarization $\mathbf{f}_0^u$.

Fix a choice of $z_1\in R$ such that $R=\mathbb{Z}+\mathbb{Z}z_1$ and define $\rho:M\hookrightarrow \mathrm{M}_2(\mathbb{Q})$ by a regular representation
\[\rho(\alpha)\begin{pmatrix} z_1 \\ 1 \end{pmatrix}=\begin{pmatrix} \alpha z_1 \\ \alpha \end{pmatrix}\; . \]
After tensoring with $\mathbb{A}$ we get $\rho:{M^{\times} \backslash M^{\times}_{\mathbb{A}}}  \hookrightarrow G(\mathbb{Q})\backslash G(\mathbb{A})$. 

We fix $g_1\in G(\mathbb{A})$ such that $g_{1,\infty}(\mathbf{i})=z_1$ and $\mathrm{det}(g_{1,\infty})=1$ while the finite places of $g_1$ will be specified shortly. We recall Lemma 3.7 of \cite{HidaNV}.
\begin{factoring}\label{factoring}
Let $\chi_m:{M^{\times} \backslash M^{\times}_{\mathbb{A}}}\to\mathbb{C}^\times$ be a Hecke character with $\chi_m|_{\mathbb{A}^\times} = \boldsymbol{\psi}_m^{-1}$ and $\chi_m(a_\infty)=a_\infty^{k+2m}$. Then $a\mapsto  \mathbf{f}_m(\rho(a)g_1) \chi_m(a)$ factors through $I_M^-= M^\times \left\backslash M^{\times}_{\mathbb{A}}\right/(\mathbb{A}^{(\infty)})^\times M^\times_\infty$ (the anticyclotomic idele class group).
\end{factoring} 

Let $I_M=M^\times \left\backslash M^{\times}_{\mathbb{A}}\right/ M^\times_\infty$ be the idele class group and choose a Haar measure $d^{\times}a$ on $M^{\times}_{\mathbb{A}}/ M^\times_\infty$  so that $\int_{\widehat{R}^\times} d^{\times}a=1$ as in Section 2.1 of \cite{HidaNV}. Taking a fundamental domain $\Phi\subset M^{\times}_{\mathbb{A}}/ M^\times_\infty$ of $I_M$ we get a measure on $I_M$ still denoted by $d^{\times}a$. 

Set
\[L_{\chi_m}(\mathbf{f}_m):=\int_{I_M}  \mathbf{f}_m(\rho(a)g_1) \chi_m(a)d^{\times}a  \]
 
In \cite{HidaNV} Hida related $L_{\chi_m}(\mathbf{f}_m)^2$ to the central critical L-value $L(\frac{1}{2},\hat{\pi}_{\mathbf{f}}\otimes \chi_m^-)$ for all arithmetic Hecke characters $\chi_m$ with $\chi_m|_{\mathbb{A}^\times} = \boldsymbol{\psi}_m^{-1}$ computing explicitly all local Euler-like factors without ambiguity. Moreover this is done under optimal assumptions on conductor of $\chi_m$, one of them being sufficient depth at non-split primes dividing conductor $N$ of $\pi_{\mathbf{f}}$.

We briefly explain Hida's recipe for a choice of $g_1$ at finite places (see Section 4 of \cite{HidaNV}) adjusted to our need.
Let $N=\prod_ll^{\nu(l)}$ be the prime factorization and let $N_{ns}=\prod_{l \, \text{non-split}}l^{\nu(l)}$ be its ``non-split'' part. 
Let $\mathfrak{C}$ denote the conductor of $\chi_m$ as above. In this section we assume that $\ell \nmid N$ and that $\chi_m$ is unramified outside $N$ and $\ell$.
Moreover, we assume that its conductor
\begin{itemize}
\item at non-split primes $l|N_{ns}$ is equal to $l^{\tilde{\nu}(l)}$, where $\tilde{\nu}(l)\geq \nu(l)$ is an arbitrary integer, and 
\item at $\ell$ is equal to $\ell^n$, where  $n\geq 1$ is an arbitrary integer.
\end{itemize}
Thus, $N_{ns}$-part of conductor is $c_{ns}= \prod_{l|N_{ns}}l^{\tilde{\nu}(l)}$, which we write $\mathfrak{C}_{N_{ns}}=c_{ns}$, and $\mathfrak{C}_{\ell}=\ell^n$.

Let $d(M)\in \mathbb{Z}$ ($d(M)<0$) be the discriminant of $M$ and set $d_0(M)=|d(M)|/4$ if $4|d(M)$ while $d_0(M)=|d(M)|$ otherwise. 
We divide the set of prime factors of $N(\mathfrak{C})d_0(M)N$ into disjoint union $A\sqcup C$ as follows. It suffices for our purpose to set $A=\{\ell\}$ when $\ell$ splits in $M$, and leave $A=\emptyset$ otherwise. Set $C=C_0\sqcup C_1$ where $C_1$ is the set of prime factors of $d_0(M)$ and $C_0=C_i\sqcup C_{sp}\sqcup C_r$ so that $C_i$ consists of primes inert in $M$, $C_r=\{2\}$ if $\mathrm{ord}_2(d(M))=2$ with $\nu(2)>2$ and $C_r=\emptyset$ otherwise. Then $C_{sp}$ consists of primes split in $M$ that are not already placed in $A$. Thus, there are two possibilities for a prime $\ell$: if it splits in $M$ it is placed in set $A$, otherwise it is placed in appropriate subset of $C$.

If $\ell$ splits in $M$ we choose a prime $\bar{\mathfrak{L}}$ over $\ell$ in $M$;  we set $\mathcal{A}=\{\bar{\mathfrak{L}}\}$, and we choose a prime $\bar{\mathfrak{l}}$ over each $l\in C_{sp}$, denoting the set of all these choices by $\mathcal{C}_{sp}=\{\bar{\mathfrak{l}}\mid l\in C_{sp}\}$. For all $l\in A\sqcup C_{sp}$ we can identify $M_l= M_{\bar{\mathfrak{l}}} \times M_{\mathfrak{l}} =\mathbb{Q}_l\times\mathbb{Q}_l$  and write $\iota_l(\alpha)=\alpha$ and $c\circ\iota_l(\alpha)=\bar{\alpha}$ where $\iota_l$ and $c\circ\iota_l$ are projections of $M_l$ to $M_{\mathfrak{l}}$ and $M_{\bar{\mathfrak{l}}}$, respectively. Note that these identifications follow reversed notation from the ones in \cite{HidaNV} due to a reason explained in Remark \ref{switch} at prime $p$ -- we proceed similarly at other split primes to keep our notation uniform. Then for  $l\in A\sqcup C_{sp}$ we specify $g_{1,\ell}$ and $g_{1,l}$  by first choosing $h_{1,\ell}\in G(\mathbb{Z}_\ell)$ and $h_{1,l}\in G(\mathbb{Z}_l)$ so that $h_{1,l}^{-1}\rho(\alpha)h_{1,l}=\bigl(\begin{smallmatrix} \bar{\alpha}  & 0 \\ 0 & \alpha \end{smallmatrix} \bigr)$; for example $h_{1,\ell}=h_{1,l}=\bigl(\begin{smallmatrix}  \bar{z}_1 & z_1 \\ 1 & 1 \end{smallmatrix} \bigr)$ will do, and then setting:
\[g_{1,\ell}= \begin{cases}
h_{1,\ell}\bigl(\begin{smallmatrix} \ell^n & 1 \\ 0 & 1 \end{smallmatrix} \bigr) & \text{if } \ell \in A \text{ is split},\\
 \bigl(\begin{smallmatrix} \ell^n & 0 \\ 0 & 1 \end{smallmatrix} \bigr) & \text{if } \ell \text{ is non-split} .\\
\end{cases}
\]
\[g_{1,l}=h_{1,l}\bigl(\begin{smallmatrix} l^{\tilde{\nu}(l)} & 0 \\ 0 & 1 \end{smallmatrix} \bigr)\text{ for } l\in C_{sp} \; .\]
Unless $l=2$ is inert in $M$, we set
\[g_{1,l}=\bigl(\begin{smallmatrix} l^{\tilde{\nu}(l)} & 0 \\ 0 & 1 \end{smallmatrix} \bigr)\; \text{for}\; l\in C_i\sqcup C_r \sqcup C_1 \backslash \{\ell\}\; .\]
If exceptionally, $2\in C$ and 2 is inert in $M$, the appropriate choice of $g_{1,l}$ for $l=2$ is given in Lemma 2.5 of \cite{HidaNV}. We chose $g_{1,\infty}\in G(\mathbb{R})$ so that $g_{1,\infty}(\mathbf{i})=z_1$ and $\mathrm{det}(g_{1,\infty})=1$. 
We set $g_{1,l}$ to be the identity matrix in $G(\mathbb{Z}_l)$ for $l\not\in A\sqcup C \sqcup\{\infty\}$ (see the proof of Proposition 2.2 in \cite{HidaNV}).

We recall Lemma 6.2 of \cite{Br2} where the above Lemma \ref{factoring} is appropriately improved under these new assumptions:
\begin{improv}\label{improv}
Let $\chi_m:{M^{\times} \backslash M^{\times}_{\mathbb{A}}}\to\mathbb{C}^\times$ be a Hecke character such that its conductor $\mathfrak{C}$ satisfies $\mathfrak{C}_{c_{ns}\ell}= c_{ns}\ell^n$ and $\chi_m|_{\mathbb{A}^\times} = \boldsymbol{\psi}_m^{-1}$ and $\chi_m(a_\infty)=a_\infty^{k+2m}$. Then $a\mapsto \mathbf{f}_m(\rho(a)g_1) \chi_m(a)$ factors through $\mathrm{Cl}^-_n:=\mathrm{Cl}_M^-(c_{ns}\ell^n)= M^\times \left\backslash M^{\times}_{\mathbb{A}}\right/(\mathbb{A}^{(\infty)})^\times\widehat{R}_{c_{ns}\ell^n}^\times M^\times_\infty$ .
\end{improv} 
Thus, if we choose a complete set of representatives $\mathfrak{a}\in \mathrm{Cl}^-_n$ and corresponding $a\in M^{\times}_{\mathbb{A}}$ such that $\widehat{\mathfrak{a}}=a\widehat{R}_{c_{ns}\ell^n}$, we immediately conclude
\begin{equation} \label{int-sum}
L_{\chi_m}(\mathbf{f}_m) =\frac{\mathop{vol}(I_M)}{|\mathrm{Cl}^-_n|}\sum_{\mathfrak{a}\in \mathrm{Cl}^-_n}\chi_m(a)\mathbf{f}_m(\rho(a)g_1)
= \frac{\varphi_{\mathbb{Q}}(c_{ns}\ell^n)}{2\varphi_M(c_{ns}\ell^n)}\sum_{\mathfrak{a}\in \mathrm{Cl}^-_n }\chi_m(a)\mathbf{f}_m(\rho(a)g_1) 
\end{equation}
Here $\mathop{vol}(I_M)=\int_{I_M} d^{\times}a=h(M)/|R^\times|$ where $h(M)$ is class number of $M$, because under chosen normalization of Haar measure we have $\int_{\widehat{R}^\times/R^\times} d^{\times}a=1/|R^\times|$. Additionally, 
\[|\mathrm{Cl}^-_n|=\frac{2h(M)\varphi_M(c_{ns}\ell^n)}{|R^\times|\varphi_\mathbb{Q}(c_{ns}\ell^n)}\]
can be read off from the exact sequence
\[ 0\longrightarrow \mathrm{Cl}_{\mathbb{Q}}(c_{ns}\ell^n)\longrightarrow \mathrm{Cl}_M(c_{ns}\ell^n)\longrightarrow \mathrm{Cl}^-_n\longrightarrow 0\]
where $\mathrm{Cl}_M(c_{ns}\ell^n)$ and $\mathrm{Cl}_{\mathbb{Q}}(c_{ns}\ell^n)$ denote ray class groups modulo $c_{ns}\ell^n$ of $M$ and $\mathbb{Q}$, respectively.

In the Section 6 of \cite{Br2}, by examining how the local components $g_{1,l}$ affect the conductor of a lattice associated to a CM point on $Sh$ in the sense of Deligne's treatment, we gave a detailed verification that if $[z_1,1]\in Sh(\mathbb{C})$ corresponds to $x(R)=\xw{R}{\mathcal{W}}$ then $[z_1,\rho(a)g_1]$, for $a$ as above, corresponds to $x(\mathfrak{a}')=\xw{\mathfrak{a}'}{\mathcal{W}}$ on $Sh$, where $\mathfrak{a}'$ is certain proper $R_{c_{ns}\ell^n}$-ideal not necessarily representing the same class in $\mathrm{Cl}^-_n$ as $\mathfrak{a}$. However, as we are taking sum and free to do so over any choice of representatives of $\mathrm{Cl}^-_n$, we have
{\allowdisplaybreaks
\begin{align} \nonumber
\left(\frac{\varphi_{\mathbb{Q}}(c_{ns}\ell^n)}{2\varphi_M(c_{ns}\ell^n)}\right)^{-1}L_{\chi_m}(\mathbf{f}_m)
&\stackrel{(\ref{int-sum})}{=} \sum_{\mathfrak{a}\in \mathrm{Cl}^-_n}\chi_m(a^{-1})\mathbf{f}_m(\rho(a^{-1})g_1) \\
&\stackrel{(\ref{der})}{=} |\mathrm{det}(g_1^{(\infty)})|_\mathbb{A}^{-m} \sum_{\mathfrak{a}\in \mathrm{Cl}^-_n}\chi_m(a^{-1})|\mathrm{det}(\rho(a^{-1}))|_{\mathbb{A}}^{-m} \delta_k^m\mathbf{f}(\rho(a^{-1})g_1) \nonumber \\
&=j(g_{1,\infty},\mathbf{i})^{-k-2m}|\mathrm{det}(g_1^{(\infty)})|_\mathbb{A}^{-m}\sum_{\mathfrak{a}\in \mathrm{Cl}^-_n}\chi_m(a^{-1})\left(|a^{-1}|_{M_\mathbb{A}}^{-m} (\delta_k^mf)(x(\mathfrak{a}),\omega_\infty(\mathfrak{a}))\right) \nonumber \\
&=j(g_{1,\infty},\mathbf{i})^{-k-2m}|\mathrm{det}(g_1^{(\infty)})|_\mathbb{A}^{-m}\sum_{\mathfrak{a}\in \mathrm{Cl}^-_n}\left(\chi_m(a^{-1})|a|_{M_\mathbb{A}}^m\right) (\delta_k^mf)(x(\mathfrak{a}),\omega_\infty(\mathfrak{a})) \nonumber
\end{align} }
whence by the Katz--Shimura rationality result (\ref{K-S}) we conclude
\begin{equation}\label{torusint} 
\tilde{C}\frac{L_{\chi_m}(\mathbf{f}_m)}{\Omega_\infty^{k+2m}}= \sum_{\mathfrak{a}\in \mathrm{Cl}^-_n}\left(\chi_m(a^{-1})|a|_{M_\mathbb{A}}^m\right) (d^mf)(x(\mathfrak{a}),\omega(\mathfrak{a})) \in \mathcal{W} 
\end{equation}
where
\begin{equation}\label{fudge}
\tilde{C}:=j(g_{1,\infty},\mathbf{i})^{k+2m}|\mathrm{det}(g_1^{(\infty)})|_\mathbb{A}^m\left(\frac{\varphi_{\mathbb{Q}}(c_{ns}\ell^n)}{2\varphi_M(c_{ns}\ell^n)}\right)^{-1}\; .
\end{equation}

We are now ready to invoke the main Theorem 4.1 of \cite{HidaNV} that relates square of $L_{\chi_m}(\mathbf{f}_m)$ to the central critical value $L(1/2,\hat{\pi}_{\mathbf{f}}\otimes \chi^-)$. For the starting Hecke newform $f_0\in S_k(\Gamma_0(N),\psi)$, let $f_0|T(l)=a(l,f_0)f_0$ for primes $l$. As usual, we define Satake parameters $\alpha_l,\beta_l\in\mathbb{C}$ by the equations $\alpha_l+\beta_l=a(l,f_0)/l^{(k-1)/2}$ and $\alpha_l\beta_l=\psi(l)$ when $l\nmid N$, while we set $\alpha_l=a(l,f_0)/l^{(k-1)/2}$ and $\beta_l=0$ when $l|N$. Then the primitive L-function is the product 
\[ L(s,\hat{\pi}_{\mathbf{f}}\otimes \chi^-)= \prod_l E_l(s)\]
of Euler factors given by
\[ E_l(s)=
\begin{cases}
\left[ (1-\frac{\alpha_l\chi_m^-(\mathfrak{l})}{l^s})(1-\frac{\alpha_l\chi_m^-(\bar{\mathfrak{l}})}{l^s})(1-\frac{\beta_l\chi_m^-(\mathfrak{l})}{l^s})(1-\frac{\alpha_l\chi_m^-(\bar{\mathfrak{l}})}{l^s})\right]^{-1} & \text{if } l=\mathfrak{l}\bar{\mathfrak{l}}\text{ splits in } M,\\
\left[ (1-\frac{\alpha_l^{2/e}\chi_m^-(\mathfrak{l})}{l^{2s/e}}) (1-\frac{\beta_l^{2/e}\chi_m^-(\mathfrak{l})}{l^{2s/e}})\right]^{-1} & \text{if } l=\mathfrak{l}^e \text{ is non-split}.
\end{cases} \]
where $\chi_m^-(\mathfrak{l})=0$ if $\mathfrak{l}$ divides the conductor $\bar{\mathfrak{C}}$ of $\chi_m^-$. 

Now we explain how starting from  $f_0$, we choose a suitable $f$ such that its arithmetic lift $\mathbf{f}$ is inside the automorphic representation $\pi_{\mathbf{f}_0}$ generated by the unitarization $\mathbf{f}_0^u$. The form $f$ is a normalized Hecke eigen-cusp form in $\pi_{\mathbf{f}_0}$ with $f|T(n)=a(n,f)f$ and $a(l,f)=a(l,f_0)$ for all primes $l$ outside $\mathrm{lcm}(N,\ell,d_0(M))$. If $\mathrm{lcm}(N,\ell,d_0(M))=N$, we set $f:=f_0$, otherwise it is possible to choose $f$ inside $\pi_{\mathbf{f}_0}$ such that for primes $l|\mathrm{lcm}(N,\ell,d_0(M))$ we have
\[ a(l,f)=
\begin{cases}
a(l,f_0) & \text{if } l|N,\\
\alpha_ll^{(k-1)/2} & \text{if } l\nmid N.\\
\end{cases}
\]
Note that $\pi_{\mathbf{f}_0}=\pi_{\mathbf{f}}$, and we write $\pi_{\mathbf{f}}$ from now.
The condition (F) of the main Theorem 4.1 of \cite{HidaNV} requires special care. Our character $\chi_m$ is of the form $\chi_m:=\lambda  \cdot \chi \cdot |\cdot|_{M_\mathbb{A}}^m$, where $\lambda$ is a fixed choice of Hecke character of $\infty$-type $(k+m,-m)$ such that $\lambda|_{\mathbb{A}^\times} = \boldsymbol{\psi}^{-1}$ and $\chi$ is any finite order anticyclotomic character of conductor $\ell^n$ -- so it factors through $\mathrm{Cl}_n^-$. We give the following
\begin{recipe} 
We choose $\lambda=\lambda_0\chi_0$, where $\lambda_0$ has infinity type $(k,0)$ and provides criticality condition $\lambda_0|_{\mathbb{A}^\times} = \boldsymbol{\psi}^{-1}$ and $\chi_0$ is arbitrary fixed choice of anticyclotomic Hecke character of $\infty$-type $(m,-m)$ and conductor $c_{ns}$. Note that anticyclotomic Hecke characters are trivial on rational adeles, hence $\lambda$ inherits criticality condition from $\lambda_0$. Concerning the choice of $\lambda_0$, the criticality assumption $\lambda_0|_{\mathbb{A}^\times}=\boldsymbol{\psi}^{-1}$ forces the conductor of $\lambda_0$ to depend on the conductor $c(\psi)$ of nebentypus $\psi$. Regardless of the choice of $\lambda_0$ we have:
\begin{itemize}
\item at primes $l\nmid c(\psi)$ Hecke character $\lambda_0$ is unramified and
\item at non-split primes $l|c(\psi)$ the conductor of $\lambda_{0,l}$ divides $l^{\mathrm{ord}_l(c(\psi))}$ hence does not exceed $l^{\tilde{\nu}(l)}$.
\end{itemize}
However, if a split prime $l|c(\psi)$ we make a subtle choice as follows. Note that $M_l^\times= M_{\bar{\mathfrak{l}}}^\times \times M_{\mathfrak{l}}^\times  =\mathbb{Q}_l^\times\times\mathbb{Q}_l^\times$ and consequently $\lambda_{0,l} = \lambda_{0,\bar{\mathfrak{l}}}\lambda_{0,\mathfrak{l}}$. Then
\begin{itemize}
\item at split primes $l|c(\psi)$ we choose $\lambda_{0,l}$ so that its conductor is supported at $\mathfrak{l}$, that is, we choose $\lambda_{0,\bar{\mathfrak{l}}}$  to be unramified and $\lambda_{0,\mathfrak{l}}$ to have conductor $\mathfrak{l}^{\mathrm{ord}_l(c(\psi))}$. 
\end{itemize}
\end{recipe}
Then $\chi_m$ has conductor $\mathfrak{C}=c_{ns}\ell^n\prod_{l\in C_{sp}}\mathfrak{l}^{\mathrm{ord}_l(c(\psi))}$
and aforementioned condition (F) is satisfied so by Theorem 4.1 of \cite{HidaNV} we have
\begin{equation}\label{wald} 
L_{\chi_m}(\mathbf{f}_m)^2=  c\frac{\Gamma(k+m)\Gamma(m+1)}{(2\pi\mathbf{i})^{k+2m+1}}E(1/2)E'(m)L^{(Nd)}(\frac{1}{2},\hat{\pi}_{\mathbf{f}}\otimes (\lambda\chi)^-)\,.
\end{equation}
In the following, the notation $[\cdot]^*$ means that the factor inside the brackets appears only if $\ell$ splits in $M$ and consequently $A=\{\ell\}$ is non-empty. Let $N_{ns}=\prod_{l \, \text{split}}l^{\nu(l)}$ be the ``split'' part of prime factorization $N=\prod_ll^{\nu(l)}$. The constant $c=c_1\cdot G\cdot v$ with 
\[c_1=[\mathrm{exp}(-\frac{2\pi\mathbf{i}}{\ell^n})]^*\sqrt{d(M)}(2\mathbf{i})^{-k-2m} (c_{ns}N_{sp}\ell^n)^{k+2m}\] 
is given by
\[v=\frac{\prod_{l\in C_{sp}}l^{\tilde{\nu}(l)}}{[\ell^n(1-\frac{1}{\ell})^3]^*\prod_{l\in C_i}l^{2\tilde{\nu}(l)}(1+\frac{1}{l})^2(1-\frac{1}{l})
\prod_{l\in C_r\cup C_1, \tilde{\nu}(l)>0}(1-\frac{1}{l})}\; ,\]
\begin{equation}\label{gauss}
G=[\chi_{m,\ell}^-(\ell^s)^{-1} G(\chi_{m,\bar{\mathfrak{L}}}^-)]^*\prod_{l\in C}\chi_{m,l}^-(l^{\nu(l)})^{-1}\prod_{l\in C_{sp},l|c(\psi)}l^{((k/2)+m)(\nu(l)-\mathrm{ord}_l(c(\psi)))}\chi_{m,\bar{\mathfrak{l}}}^-(l^{\nu(l)-\mathrm{ord}_l(c(\psi))})G(\chi_{m,\bar{\mathfrak{l}}}^-) \;,
\end{equation}
where $\chi_{m,\ell}^-=\chi_m^-|_{\mathbb{Q}_\ell^\times}$, $\chi_{m,\bar{\mathfrak{l}}}^-=\chi_m^-|_{M_{\bar{\mathfrak{l}}}^\times}$, $G(\chi_{m,\bar{\mathfrak{L}}}^-)$ and $G(\chi_{m,\bar{\mathfrak{l}}}^-)$ is Gauss sum of $\chi_{m,\bar{\mathfrak{L}}}^-=\chi_m^-|_{M_{\bar{\mathfrak{L}}}^\times}$ and $\chi_{m,\bar{\mathfrak{l}}}^-=\chi_m^-|_{M_{\bar{\mathfrak{l}}}^\times}$, respectively, 
and the modification Euler factors are given by
\begin{equation} \label{euler1}
E(1/2)=\prod_{\mathfrak{l}|l\in C_{sp}}\left(1-\frac{\chi_m^-(\mathfrak{l})\alpha_l}{N(\mathfrak{l})^{1/2}}\right)^{-1}\prod_{\mathfrak{l}|d(M)}\left(1-\frac{\chi_m^-(\mathfrak{l})\alpha_l}{N(\mathfrak{l})^{1/2}}\right)^{-1}\left(1-\frac{\chi_m^-(\mathfrak{l})\beta_l}{N(\mathfrak{l})^{1/2}}\right)^{-1}\, , 
\end{equation}
\begin{equation}\label{euler2}
E'(m)=\frac{
\prod_{l\in C_{sp},l|c(\psi)} \frac{\alpha_l^{\nu(l)-\mathrm{ord}_l(c(\psi))}}{l^{(\nu(l)-\mathrm{ord}_l(c(\psi)))(m+(k-1)/2)}}
\prod_{l\in C_{sp},l\nmid c(\psi)}\alpha_l^{\nu(l)}l^{\nu(l)/2}\chi_m^-(\bar{\mathfrak{l}}^{\nu(l)})\left( 1-\frac{1}{\alpha_ll^{1/2}\chi_m^-(\bar{\mathfrak{l}})}\right)}
{\prod_{l\in C_{sp}}\alpha_l^{\nu(l)}l^{\nu(l)/2}\chi_m^-(\mathfrak{l}^{\nu(l)})\left(1-\frac{1}{\alpha_ll^{1/2}\chi_m^-(\mathfrak{l})}\right)}
\, .
\end{equation}
Then (\ref{torusint}) and (\ref{wald}) suggest that we normalize L-value as follows
\begin{equation}
L^{\mathrm{alg}}(\frac{1}{2},\hat{\pi}_{\mathbf{f}}\otimes (\lambda\chi)^-)):= G\frac{\Gamma(k+m)\Gamma(m+1)}{\pi^{k+2m+1}}E(1/2)E'(m)\frac{L^{(Nd)}(\frac{1}{2},\hat{\pi}_{\mathbf{f}}\otimes (\lambda\chi)^-)}{{\Omega_\infty^{2(k+2m)}}}\,.
\end{equation}
As we are going to study $L^{\mathrm{alg}}(\frac{1}{2},\hat{\pi}_{\mathbf{f}}\otimes (\lambda\chi)^-))$ modulo $p$ by utilizing (\ref{torusint}), we are clearly going to exclude $p$ from the primes that divide fudge factors $c_1$ and $v$ above, and $\tilde{C}$ given in (\ref{fudge}). More concretely, let $\mathcal{S}(N,\ell)$ be the finite set of prime divisors of elements of
\begin{equation}\label{fudgeset}
 \{N, \ell-1\} \cup \{l-1\,:\text{ prime }l\mid N \text{ is ramified in }M\} \cup \{l-1,l+1\,:\text{ prime }l\mid N \text{ is inert in } M\}\,.
\end{equation}

\section{Zariski density of CM points on Shimura curves}\label{sec:zariski}
Each integral $R$-ideal $\mathfrak{A}$ prime to $c_{ns}\ell$ induces a unique proper $R_{c_{ns}\ell^n}$-ideal $\mathfrak{a}_n=\mathfrak{A}\cap R_{c_{ns}\ell^n}$ and after taking the inverse limit $\mathfrak{a}=\underleftarrow{\lim}_n\mathfrak{a}_n$ we obtain an element of $\mathrm{Cl}^-_\infty$. We are able to embed the group $\mathrm{Cl}^{\mathrm{alg}}$ of fractional ideals of $R$ prime to $c_{ns}\ell $ as a subgroup of $\mathrm{Cl}^-_\infty$ by extending this procedure to fractional ideals in obvious way.

In Section \ref{sec:heckecmpts}, starting from a pair $(X(R), \eta^{(p)}(R))$, we associated to a certain representative of each proper ideal class $\mathfrak{a}\in \mathrm{Cl}^-_n$ a $\mathcal{W}$-point $x(\mathfrak{a})=(X(\mathfrak{a}),\eta^{(p)}(\mathfrak{a}))$ on $Sh^{(p)}$. Let $\mathfrak{a}$ be a proper $R_{c_{ns}\ell^n}$-ideal. Note that each $\alpha\in R_{(p)}^\times$ induces an isomorphism
\[(X(\mathfrak{a}),\eta^{(p)}(\mathfrak{a}))\cong (X(\alpha\mathfrak{a}),\eta^{(p)}(\alpha \mathfrak{a}))\,.\]
Thus the isomorphism class of pair $x(\mathfrak{a})$ depends only on the proper ideal class of $\mathfrak{a}$ in $\mathrm{Cl}^-_n$. We define $\rho:R^\times_{(p)}\to G(\mathbb{A}^{(p\infty)})$ via $\alpha \eta^{(p)}(R) = \eta(R) \circ \rho(\alpha)$. Then $x=(X(R), \eta^{(p)}(R)) \in Sh^{(p)}$ is fixed by $\rho(\alpha)$.

Let $V_{{/\bar{\mathbb{F}}_p}}$ be the irreducible component of $Sh^{(p)}_{{/\bar{\mathbb{F}}_p}}$ containing a point $x(\mathfrak{a})_{{/\bar{\mathbb{F}}_p}}$. In other words, $V_{{/\bar{\mathbb{F}}_p}}$ is the special fiber of geometrically connected component $V^{(p)}_{/\mathcal{W}}$ of $Sh^{(p)}$ introduced in the Section \ref{sec:dbf}. Given any infinite sequence of integers $\underline{n}=\{n_j\}_1^\infty$ we set
\[\Xi=\Xi_{\underline{n}}:=\{x(\mathfrak{a})\in V(\bar{\mathbb{F}}_p)\mid\mathfrak{a}\in\mathrm{Ker}(\pi_{n_{j},n_1}),\; j=1,2,\ldots\}\,.\]
Elliptic curves sitting over points in $\Xi$ are non-isomorphic by a result of Deuring (\cite{Deu}) which assures that the isomorphism class of an elliptic curve over $\bar{\mathbb{F}}_p$ is determined by the action of relative Frobenius map on its $\ell$-adic Tate module. Thus, the set $\Xi$ is infinite  and we  record the following obvious
\begin{fact} Since $\mathrm{dim}V_{{/\bar{\mathbb{F}}_p}}=1$, the set $\Xi$ is Zariski dense in $V(\bar{\mathbb{F}}_p)$.
\end{fact}
For a finite subset $\mathscr{Q}\subset \mathrm{Cl}^-_\infty$ we set
\[\Xi^{\mathscr{Q}}:=\{\left(x(\delta(\mathfrak{a}))\right)_{\delta\in\mathscr{Q}}\in V^{\mathscr{Q}}\mid x(\mathfrak{a})\in\Xi \} \]
where $V^{\mathscr{Q}}=\prod_{\delta\in\mathscr{Q}}V$ and $\delta(\mathfrak{a})$ is a proper ideal class representative of the product of $\pi_{\infty,n_j}(\delta)$ and $\mathfrak{a}$ in $\mathrm{Cl}^-_{n_j}$.
We recall Proposition 2.8 of \cite{HidaDwork} since it is of instrumental use for us.
\begin{density}\label{density}
If $\mathscr{Q}$ is finite and injects into $\mathrm{Cl}^-_\infty/\mathrm{Cl}^{\mathrm{alg}}$, the subset $\Xi^{\mathscr{Q}}$ is Zariski dense in  $V^{\mathscr{Q}}_{/\bar{\mathbb{F}}_p}$.
\end{density}
A detailed proof is given in \cite{HidaDwork} and we just give a brief sketch noting that its key ingredient is an instance of Chai's Hecke orbit principle (see \cite{Ch} Section 8) stated below. Let $\mathscr{Q}=\{\delta_1,\ldots,\delta_h\}$. The torus $T=R^\times_{(p)}/\mathbb{Z}^\times_{(p)}$ acts diagonally on $V^h$ via $\prod \rho$. Let $Z_1$ be Zariski closure of $\Xi^{\mathscr{Q}}$ in $V^h$. The group $T_1=\{\alpha\in T| \alpha\equiv 1 \,(\mathrm{mod} \, \ell^{n_1})\}$ leaves $Z_1$ stable, as it just permutes $\Xi^{\mathscr{Q}}$ by $\rho(\alpha)(x(\mathfrak{a}))=x(\alpha\mathfrak{a})$. If $Z_0$ denotes the irreducible component of $Z_1$ containing $(x(R_{c_{ns}\ell^{n_1}}),\ldots,x(R_{c_{ns}\ell^{n_1}})$, then the stabilizer $T_0$ of $Z_0$ is of finite index in $T_1$ and its $p$-adic closure is open in $T$. Thus the following theorem applies.
\begin{chai}\label{chai}
\emph{(\cite{minv} Corollary 3.19)} Let $Z_0$ be an irreducible subvariety of  $V^m, m\geq 1,$ containing a fixed point of $T$. If there exists a subgroup $T_0\subset T$ whose $p$-adic closure is open in $R_p^\times/\mathbb{Z}_p^\times$ that stabilizes $Z_0$, then $Z_0$ is a Shimura subvariety of $V^m$.
\end{chai}
As $Z_0$ is nontrivial Shimura subvariety of $V^h$ (i.e. not of the form $V^{h-1}\times\{x\}$ for a fixed CM point $x$), this leaves two possibilities: either 
$Z_0=V^h$ or, after a permutation of factors of $V^h$, $Z_0$ is a Shimura subvariety of $V^{h-2}\times \Delta_{\beta,\beta'}$, where for some $\beta,\beta'\in R_{(p)}$ we define a diagonal by
\[ \Delta_{\beta,\beta'}=\{ (x\circ\rho(\beta),x\circ\rho(\beta'))| x \in V\}=\{(x,x\circ\rho(\beta^{-1}\beta'))|x\in V\}\subset V^2\,.\]
However, the second possibility imposes $\delta_{h-1}/\delta_h=\beta^{-1}\beta'\in M^\times$, whence $\delta_{h-1}\mathrm{Cl}^{\mathrm{alg}}=\delta_h \mathrm{Cl}^{\mathrm{alg}}$ and Theorem \ref{density} follows.

Let $\mathbf{\mathcal{C}}_{\Xi}$ denote the space of functions on $\Xi$ with values in $\mathrm{P}^1(\bar{\mathbb{F}}_p)=\bar{\mathbb{F}}_p\sqcup\{\infty\}$. The  class group $\mathrm{Cl}^-_\infty$ acts on $\mathbf{\mathcal{C}}_{\Xi}$ by left translation. By virtue of Zariski density of $\Xi$ in $V$ we can embed the function field of $V$ into $\mathbf{\mathcal{C}}_{\Xi}$. Then Theorem \ref{density} has the following 
\begin{linebundle}\label{linebundle}\emph{(\cite{HidaDwork} Corollary 2.9)} Let $\mathscr{L}$ be a line bundle over $V_{{/\bar{\mathbb{F}}_p}}$. Then for a finite set $\mathscr{Q}\subset \mathrm{Cl}^-_\infty$ that injects into $\mathrm{Cl}^-_\infty/\mathrm{Cl}^{\mathrm{alg}}$ and a set $\{f_\delta\in \mathscr{L}\mid \delta \in \mathscr{Q}\}$ of non-constant global sections $f_\delta$ of $\mathscr{L}$ finite at $\Xi$, the functions $f_\delta\circ\delta$, $\delta \in \mathscr{Q}$, are linearly independent in $\mathbf{\mathcal{C}}_{\Xi}$.
\end{linebundle}

\section{Non-vanishing of $L$-values modulo $p$}
In this section we prove Theorem \ref{main} following closely Section 3.4 of \cite{HidaDwork}, particularly the proof of Theorem 3.2 there. That being said, we avoid constructing anticyclotomic measure, in order not to impose condition $a(\ell,f)\not =0$. Before proceeding to the proof, we need a technical lemma. To this end, for $\alpha\in \mathrm{GL}_2^+(\mathbb{R})$ and a classical modular form $f\in S_k(\Gamma_0(N),\psi)$, we define
\[f\|_k\alpha=\mathrm{det}(\alpha)^{k/2}f(\alpha(z))j(\alpha,z)^{-k}\]
Note that the operator $\|_k$ depends on the weight of the form, but since the weight will always be clear from the context we shall write it as $\|$. We use non-standard notation to distinguish it from the isogeny action defined in (\ref{isogmf}). In the following lemma, we investigate the effect of the isogeny action defined in Section \ref{sec:isog}, on $q$-expansions of Katz $p$-adic derivatives of modular forms. As in (\ref{isogmf}) of Section \ref{sec:isog}, for $f\in V(N;W)$ and a prime $r$, we consider action  
\[f|[r] = f|\bigl(\begin{smallmatrix} 1 & 0 \\ 0 & r \end{smallmatrix} \bigr) \]
where $\bigl(\begin{smallmatrix} 1 & 0 \\ 0 & r \end{smallmatrix} \bigr)\in G(\mathbb{Q}_r)\subset G(\mathbb{A}^{(\infty)})$ is concentrated at $r$. If $r=r_1\cdot\ldots\cdot r_n$ is a square-free product of primes, we define
\[f|[r]=f|[r_1]|\ldots |[r_n]\,.\]
In this section we use $\mathbf{q}$ to denote the variable in $\mathbf{q}$-expansion of a modular form in order to avoid abuse of notation in the proof of the main theorem.
Then we have the following
\begin{isogmaneuver}\label{isogmaneuver}
For $f\in S_k(\Gamma_0(N),\psi)$ and a square-free integer $r=r_1\cdot\ldots\cdot r_n$ prime to $N$, we have 
\[(d^mf)|[r]=\psi(r)r^{k/2+m}d^m\left(f\| \bigl(\begin{smallmatrix} 1 & 0 \\ 0 & r \end{smallmatrix} \bigr)^{-1}\right)\,.\]
In particular, the $\mathbf{q}$-expansion of $(d^mf)|[r]$ is given by
\[ \psi(r)r^{k+2m}\sum_{n\geq 0}n^m a(n,f)\mathbf{q}^{nr}\,. \]
\end{isogmaneuver}
\begin{proof}
Let $x=(X,\eta^{(p)}, \eta^{ord})$ denote a general test object that gives rise to a point in the ordinary locus of $Sh$, and let $\omega_p(x)$ be the invariant differential induced by $\eta^{ord}$ as in Section \ref{diffop}. 

We first prove the assertion when $r$ is a prime. If we set $\alpha=\bigl(\begin{smallmatrix} 1 & 0 \\ 0 & r \end{smallmatrix}\bigr)\in G(\mathbb{Q})\subset G(\mathbb{A})$ then the crux of the proof is the following identity that holds for all $m\geq 0$:
\begin{equation}\label{subtle}
\delta_k^mf((x,\omega_\infty(x))\circ \alpha_r) = \psi(r)r^{k/2+m}\delta_k^m (f\|\alpha_\infty^{-1})(x ,\omega_\infty(x))\;. 
\end{equation}
To verify it, recall that if $x\in Sh$ corresponds to $[z,g^{(\infty)}]$ in $Sh(\mathbb{C})$, for some $z\in\mathfrak{X}$ and $g^{(\infty)}\in  G(\mathbb{A}^{(\infty)})$, then by definition
\[\delta_k^mf(x ,\omega_\infty(x))=\delta_k^mf([z,g^{(\infty)}])=\delta_k^m\mathbf{f}(g)j(g_\infty, \mathbf{i})^{k+2m} \]
where $g_\infty \in G(\mathbb{R})$ is such that $g_\infty(\mathbf{i})=z$ and $g=g^{(\infty)}g_\infty$.
Note that we need to check identity over $Sh/\widehat{\Gamma}_1(N)$ only, so without loss of generality we may assume that $g^{(\infty)}=1$.
{\allowdisplaybreaks \begin{align} \nonumber
\delta_k^mf((x ,\omega_\infty(x))\circ \alpha_r)&=\delta_k^m f([z,1]\circ \alpha_r) \\ 
&= \delta_k^m f([z,\alpha_r]) \nonumber \\
&= \delta_k^m\mathbf{f}( \alpha_r g_\infty ) j(g_\infty, \mathbf{i} )^{k+2m} \nonumber \\
&= \delta_k^m\mathbf{f}\left(\alpha (\alpha^{(r\infty)})^{-1} \alpha_\infty^{-1}g_\infty \right) j(g_\infty, \mathbf{i} )^{k+2m}  \nonumber \\
&= \boldsymbol{\psi}(\alpha^{(r\infty)})^{-1} \delta_k^m\mathbf{f} \left(\alpha_\infty^{-1}g_\infty \right) j(g_\infty, \mathbf{i} )^{k+2m} \nonumber \\
&= \psi(r) \delta_k^m f (\alpha_\infty^{-1}g_\infty(\mathbf{i}))j(\alpha_\infty^{-1}g_\infty,\mathbf{i})^{-k-2m} j(g_\infty, \mathbf{i} )^{k+2m} \nonumber \\
&= \psi(r) \delta_k^m f (\alpha_\infty^{-1}(z))j(\alpha_\infty^{-1},z)^{-k-2m}  \nonumber \\
&= \psi(r) r^{k/2+m} (\delta_k^m f)\|\alpha_\infty^{-1}(z) \nonumber \\
&= \psi(r) r^{k/2+m} \delta_k^m (f\|\alpha_\infty^{-1})([z,1]) \nonumber \\
&= \psi(r) r^{k/2+m} \delta_k^m (f\|\alpha_\infty^{-1})(x ,\omega_\infty(x)) \nonumber 
\end{align} }
as desired. Here we used a general fact $(\delta_k^mf)\|\alpha_\infty^{-1}=\delta_k^m(f\|\alpha_\infty^{-1})$ and an obvious fact that for this particular $\alpha$ we have $(\alpha^{(r\infty)})\in\widehat{\Gamma}_0(N)$.

By the Katz--Shimura rationality result (\ref{K-S}) we have
{\allowdisplaybreaks \begin{align} \nonumber 
\frac{d^mf((x,\omega_p(x))\circ\alpha_r)}{\Omega_p^{k+2m}} &= \frac{\delta_k^mf((x ,\omega_\infty(x))\circ \alpha_r)}{\Omega_\infty^{k+2m}}\\
&=\psi(r) r^{k/2+m}\frac{\delta_k^m (f\|\alpha_\infty^{-1})(x ,\omega_\infty(x))}{\Omega_\infty^{k+2m}} \nonumber \\ 
&=  \psi(r) r^{k/2+m}\frac{d^m (f\|\alpha_\infty^{-1})(x ,\omega_p(x))}{\Omega_p^{k+2m}} \nonumber
\end{align} }
which yields
\[ d^mf((x,\omega_p(x))\circ\alpha_r) =  \psi(r) r^{k/2+m}d^m (f\|\alpha_\infty^{-1}) \text{ for all } m\geq 0\,.\]
The assertion about the $\mathbf{q}$-expansion of $(d^mf)|[r]$ now easily follows from (\ref{diffqexp}).

To prove the lemma for an arbitrary square-free integer $r$, by induction it suffices to verify it when $r=r_1r_2$ is a product of two primes. Indeed,
{\allowdisplaybreaks \begin{align} \nonumber 
d^mf|[r] = d^mf|[r_1]|[r_2] &= \psi(r_1)r_1^{k/2+m}d^m\left(f\| \bigl(\begin{smallmatrix} 1 & 0 \\ 0 & r_1 \end{smallmatrix} \bigr)^{-1}\right)|[r_2] \nonumber \\
&=\psi(r_1)r_1^{k/2+m}\psi(r_2)r_2^{k/2+m}d^m\left(f\| \bigl(\begin{smallmatrix} 1 & 0 \\ 0 & r_1 \end{smallmatrix} \bigr)^{-1} \bigl(\begin{smallmatrix} 1 & 0 \\ 0 & r_2 \end{smallmatrix} \bigr)^{-1}\right) \nonumber \\
&=\psi(r)r^{k/2+m}d^m\left(f\| \bigl(\begin{smallmatrix} 1 & 0 \\ 0 & r \end{smallmatrix} \bigr)^{-1}\right) \nonumber
\end{align} }
and the claimed effect on the $\mathbf{q}$-expansion easily follows.
\end{proof}

Now we are ready to prove Theorem \ref{main}.
\begin{proof}
Classical modular forms are defined over a number field and we may assume that $f$ is defined over a localization $\mathcal{V}$ of the integer ring in a number field $E$. We take a finite extension of $W$ generated by $\mathcal{V}$ and the values of arithmetic Hecke character $\lambda$ and, abusing the symbol, we keep denoting it $W$. We write $\mathfrak{P}$ for the prime ideal of $W$ corresponding to $\iota_p$. The considered values $d^mf(\mathfrak{a})$ are algebraic and $\mathfrak{P}$-integral over $\mathcal{V}$ by results of Shimura and Katz (\cite{Sh75} and \cite{Ka}).

We fix a decomposition $\mathrm{Cl}^-_{\infty}=\Delta\times\Gamma$ where $\Gamma$ is a torsion free subgroup topologically isomorphic to $\mathbb{Z}_\ell$ and $\Delta$ is a finite group. Denote by $\mathbb{F}_{p}[f,\lambda]$ the finite subfield of $\bar{\mathbb{F}}_p$ generated by values of $\lambda\; \mathrm{mod}\;\mathfrak{P}$ and all $\ell|\Delta|$-th roots of unity over $\mathcal{V}/\mathfrak{P}\cap\mathcal{V}$. Similarly, if $\chi:\mathrm{Cl}^-_n\to\bar{\mathbb{F}}_p^\times$ is a character, denote by $\mathbb{F}_{p}[f,\lambda](\chi)$ the finite extension of $\mathbb{F}_{p}[f,\lambda]$ generated by values of $\chi$. 

As $\Delta$ is finite group, it suffices to fix a branch character $\nu:\Delta \to \bar{\mathbb{F}}_p^\times$ and, aiming for contradiction, suppose that 
\[L^{\mathrm{alg}}(\frac{1}{2},\hat{\pi}_{\mathbf{f}}\otimes (\lambda\chi_j)^-) \equiv 0 \; (\mathrm{mod}\, \mathfrak{P})\]
for infinitely many characters $\chi_j: \mathrm{Cl}^-_{n_j} \to \bar{\mathbb{F}}_p^\times$ such that $\chi_j|_\Delta=\nu$, where $\{n_j\}_{j=1}^\infty$ is an infinite sequence of integers. Suppose that prime $p>2$ is outside finite set $\mathcal{S}(N,\ell)$ given by (\ref{fudgeset}).
Then by plugging $\chi_m:=\lambda  \cdot \chi_j \cdot |\cdot|_{M_\mathbb{A}}^m$ in (\ref{torusint}) and (\ref{wald}), we have
\[\sum_{\mathfrak{a}\in\mathrm{Cl}_{n_j}}\lambda\chi_j(\mathfrak{a}^{-1})d^mf(x(\mathfrak{a}))=0 \text{ in } \bar{\mathbb{F}}_p\,. \]
Moreover, for each $\sigma \in \mathrm{Gal}(\bar{\mathbb{F}}_p/\mathbb{F}_{p}[f,\lambda])$ we have 
\begin{equation}\label{galois}
\sum_{\mathfrak{a}\in\mathrm{Cl}_{n_j}}(\lambda\chi_j)^\sigma(\mathfrak{a}^{-1})d^mf(x(\mathfrak{a}))=0 \,.
\end{equation}
Indeed, by Shimura's reciprocity law (\cite{ACM} 26.8 and \cite{PAF} 2.1.4), we have
\[\Phi(d^mf(x(\mathfrak{a})))=d^mf(x(\mathfrak{p}^{-1}\mathfrak{a}))\]
for the Frobenius map $\Phi(x)=x^p$ for $x\in \bar{\mathbb{F}}_p$. Thus, if $\sigma=\Phi^n$ for a positive integer $n$, we have
\[ \Phi^n \left( \sum_{\mathfrak{a}\in\mathrm{Cl}_{n_j}}\lambda\chi_j(\mathfrak{a}^{-1})d^mf(x(\mathfrak{a})) \right)= (\lambda\chi_j)^\sigma(\mathfrak{p}^n)\sum_{\mathfrak{a}\in\mathrm{Cl}_{n_j}}(\lambda\chi_j)^\sigma(\mathfrak{a}^{-1})d^mf(x(\mathfrak{a}))\,. \]
Consider the trace map from the field $\mathbb{F}_{p}[f,\lambda](\chi_j)$ to $\mathbb{F}_{p}[f,\lambda]$ given by 
\[\mathrm{Tr}_{\mathbb{F}_{p}[f,\lambda](\chi_j)/\mathbb{F}_{p}[f,\lambda]}(\xi)=\sum_{\sigma\in\mathrm{Gal}(\mathbb{F}_{p}[f,\lambda](\chi_j)/\mathbb{F}_{p}[f,\lambda])}\sigma(\xi)\]
for $\xi\in\mathbb{F}_{p}[f,\lambda](\chi_j)$ and note that if $\mathbb{F}_{p}[f,\lambda]^\times\cap\mu_{\ell^\infty}=\mu_{\ell^s}$ we have 
\[ \mathrm{Tr}_{\mathbb{F}_{p}[f,\lambda](\chi_j)/\mathbb{F}_{p}[f,\lambda]}(\chi_j(x))=
\begin{cases}
\ell^{n_j-s}\chi_j(x) & \text{if }\chi_j(x)\in \mathbb{F}_{p}[f,\lambda],\\
0 & \text{otherwise}.
\end{cases} \]
If $\langle \cdot \rangle : \mathrm{Cl}^-_\infty \to \Gamma$ denotes the natural projection, from (\ref{galois}) we conclude
\begin{equation}\label{medgal}
\sum_{\mathfrak{a}\in\mathrm{Cl}^-_{n_j}:\langle\mathfrak{a}\rangle\in \chi_j^{-1}(\mu_{\ell^s})}\lambda\chi_j(\mathfrak{a}^{-1})d^mf(x(\mathfrak{a})) = 0 \,.
\end{equation}
Note that here $\langle\mathfrak{a}\rangle\in \chi_j^{-1}(\mu_{\ell^s})$ if and only if $\langle\mathfrak{a}\rangle\in \Gamma^{\ell^{n_j-s}}/\Gamma^{\ell^{n_j}}$. Moreover, if $\Gamma_j$ denotes the image of $\Gamma$ in $\mathrm{Cl}^-_{n_j}$, as the summands are independent of the choice of a proper ideal class representative in $\mathrm{Cl}^-_{n_j}$, we conclude from (\ref{medgal}) that for every $y\in\Gamma_j$ we have
\begin{equation}\label{newgalois}
\sum_{\mathfrak{a}\in\mathrm{Cl}^-_{n_j}:\langle\mathfrak{a}\rangle\in y\chi_j^{-1}(\mu_{\ell^s})}\lambda\chi_j(\mathfrak{a}^{-1})d^mf(x(\mathfrak{a})) = 0 \,.
\end{equation}

The group $\Delta^{\mathrm{alg}}=\Delta\cap \mathrm{Cl}^{\mathrm{alg}}$ is generated by prime ideals of $M$ non-split over $\mathbb{Q}$ and we can choose a complete representative set for $\Delta^{\mathrm{alg}}$ consisting of product of prime ideals of $M$ outside $N$, $p$ and $\ell$. In \cite{HidaDwork}, the author chose this set as $\{\mathfrak{r}'^{-1} \mid r'\in\mathscr{R}'\}$, where $\mathscr{R}'$ was made of square-free products of rational primes outside $N$ and $\ell$ that are ramified in $M$ and $\mathfrak{r}'$ is a unique ideal in $M$ such that $\mathfrak{r}'^2=r'$. Thus, $\{\mathfrak{r}'\mid r'\in\mathscr{R}'\}$ was a complete representative set for $2$-torsion elements in $\mathrm{Cl}_M$, the class group of $M$. We alter this choice by noting that the set of split primes in $M$ prime to any given rational integer has positive density in the set of prime ideals of $M$ by the \v Cebotarev density theorem. In other words, in the ray class group of $M$ of an arbitrary conductor, each class contains infinitely many split prime ideals prime to any given rational integer. Thus, we choose a complete representative set for $\Delta^{\mathrm{alg}}$ as $\{\mathfrak{r}^{-1} \mid r\in\mathscr{R}\}$ by finding a split prime ideal $\mathfrak{r}$ representing a class of $\mathfrak{r}'$ as above. Thus, our $\mathscr{R}$ is made of square-free products of rational primes $r$ outside $N\ell p$ that are split in $M$ and $\mathfrak{r}$ is a choice of ideal in $M$ such that $r=\mathfrak{r}\bar{\mathfrak{r}}$.

We also choose a complete set of representatives $\mathscr{Q}$ for $\mathrm{Cl}^-_{\infty}/\Gamma\Delta^{\mathrm{alg}}$ consisting of prime ideals $\mathfrak{q}$ of $M$ that are split over $\mathbb{Q}$ and prime to $p$ and $\ell$. Then $\mathrm{Cl}^-_\infty=\bigsqcup_{\mathfrak{q},\mathfrak{r}}[\mathfrak{q}^{-1}\mathfrak{r}^{-1}]\Gamma$ and we can rewrite (\ref{newgalois}) as
\[\sum_{\mathfrak{q}\in\mathscr{Q}}\sum_{\mathfrak{r}\in\mathscr{R}}\nu(\mathfrak{r}\mathfrak{q})\sum_{\mathfrak{a}\in  y\chi_j^{-1}(\mu_{\ell^s})}\chi_j(\mathfrak{a}^{-1}\langle \mathfrak{q} \rangle)\lambda(\mathfrak{a}^{-1}\mathfrak{r}\mathfrak{q})d^mf(x(\mathfrak{q}^{-1}\mathfrak{r}^{-1}\mathfrak{a}))=0
\]
If we set
\begin{equation}\label{average}
d^mf^\nu =  \sum_{r\in\mathscr{R}}\lambda(\mathfrak{r})\nu(\mathfrak{r})(d^mf)|[r]\;.
\end{equation}
then the above identity becomes 
\begin{equation} \label{trace}
\sum_{\mathfrak{q}\in\mathscr{Q}}\nu(\mathfrak{q})\sum_{\mathfrak{a}\in  y\chi_j^{-1}(\mu_{\ell^s})} \chi_j(\mathfrak{a}^{-1})\lambda^{-1}d^mf^\nu(x(\mathfrak{q}^{-1}\langle\mathfrak{q}\rangle^{-1} \mathfrak{a}))=0\,.
\end{equation}

Fix $\mathfrak{q}\in\mathscr{Q}$. From (\ref{proj}) we know that $\{x(\mathfrak{a})|\mathfrak{a}\in y\chi_j^{-1}(\mu_{\ell^s})\}$ is given by $\{x(\mathfrak{a}_0)\circ\bigl(\begin{smallmatrix} 1 & \frac{u}{\ell^s} \\ 0 & 1 \end{smallmatrix} \bigr)|u\in \mathbb{Z}/\ell^s\mathbb{Z}\}$ where $\mathfrak{a}_0\in y\chi_j^{-1}(\mu_{\ell^s})$ is any fixed member. In particular, we can identify $y\chi_j^{-1}(\mu_{\ell^s})$ with $\mathbb{Z}/\ell^s\mathbb{Z}$ via $\mathfrak{a}\mapsto u\;\mathrm{mod}\;\ell^s$. Note that, after tacitly assuming that $n_j\geq 2s$ (which could be achieved by passing to a suitable subsequence if necessary), this is tantamount to identifying multiplicative group $\Gamma^{\ell^{n_j-s}}/\Gamma^{\ell^{n_j}}$ with additive one $\mathbb{Z}/\ell^s\mathbb{Z}$ by $1+\ell^{n_j-s}u \mapsto u$.
If we choose a primitive $\ell^s$-th root of unity $\zeta=\mathrm{exp}(2\pi i /\ell^s)$ and $\mathfrak{a}_y\in y\chi_j^{-1}(\mu_{\ell^s})$, we can write $\chi_j(\bigl(\begin{smallmatrix} 1 & \frac{u}{\ell^s} \\ 0 & 1 \end{smallmatrix} \bigr)^{-1}\mathfrak{a}_y)=\chi_j(u)=\zeta^{-uv_j}$ for some $v_j\in (\mathbb{Z}/\ell^s\mathbb{Z})^\times$, independent of $y\in \Gamma_j$, that initially depends on $\chi_j$. However, we may assume that $v_j$'s are constant $v$ by resorting to a suitable subsequence of $\{\chi_j\}_{j=1}^\infty$. Then, using the description of the operators $[r]$ and $[q]$ in Section \ref{sec:isog}, the inner sum of (\ref{trace}) is equal to
\[ \sum_{u\,\mathrm{mod}\,\ell^s}\zeta^{uv}\lambda^{-1}(d^mf^\nu)|[q]|\bigl(\begin{smallmatrix} 1 & \frac{u}{\ell^s} \\ 0 & 1 \end{smallmatrix} \bigr)\left(x(\langle\mathfrak{q}\rangle^{-1}\mathfrak{a}_y)\right)\;.\]
Note here that $\bigl(\begin{smallmatrix} 1 & \frac{u}{\ell^s} \\ 0 & 1 \end{smallmatrix} \bigr) \in G(\mathbb{Q}_\ell)\subset G(\mathbb{A}^{(\infty)})$ concentrated at $\ell$, and $\bigl(\begin{smallmatrix} 1 & 0 \\ 0 & q \end{smallmatrix} \bigr) \in G(\mathbb{Q}_q)\subset G(\mathbb{A}^{(\infty)})$ concentrated at $q$, commute.
Setting $g_{\mathfrak{q}}:= \sum_{u\,\mathrm{mod}\,\ell^s}\zeta^{uv}(d^mf^\nu)|[q]|\bigl(\begin{smallmatrix} 1 & \frac{u}{\ell^s} \\ 0 & 1 \end{smallmatrix} \bigr)$, (\ref{trace}) becomes
\[\sum_{\mathfrak{q}\in\mathscr{Q}}\nu(\mathfrak{q})\lambda^{-1}g_{\mathfrak{q}}\left(x(\langle\mathfrak{q}\rangle^{-1}\mathfrak{a})\right)=0\;.\]
Note that $\langle \mathfrak{q} \rangle$, $\mathfrak{q}\in\mathscr{Q}$, are all distinct in $\mathrm{Cl}^-_\infty / \mathrm{Cl}^{\mathrm{alg}}$ by our choice, so the set $\Xi= \Xi_{\underline{n}}$ for $\underline{n}=\{n_j\}_1^\infty$ (defined in Section \ref{sec:zariski}) is Zariski dense in $V^{\mathscr{Q}}_{/\bar{\mathbb{F}}_p}$ by Theorem \ref{density}. Then Corollary \ref{linebundle} furnishes $g_{\mathfrak{q}}=0$.

Note that for a classical modular form $f_1$ of arbitrary level and weight $k$ one has identity
\[(d^mf_1)|\bigl(\begin{smallmatrix} 1 & \frac{u}{\ell^s} \\ 0 & 1 \end{smallmatrix} \bigr)=d^m\left( f_1\| \bigl(\begin{smallmatrix} 1 & \frac{u}{\ell^s} \\ 0 & 1 \end{smallmatrix} \bigr)^{-1}\right)\]
which can be verified by literally the same argument as in proof of Lemma \ref{isogmaneuver} -- the appropriate incarnation of identity (\ref{subtle}) has no scalar factors due to the fact that matrix here is unipotent. Then fixing $\mathfrak{q}_0\in \mathscr{Q}$ and using Lemma \ref{isogmaneuver} we have
{\allowdisplaybreaks \begin{align} \nonumber 
g_{\mathfrak{q}_0} &=\sum_{u\,\mathrm{mod}\,\ell^s}\zeta^{uv}(d^mf^\nu)|[q_0]|\bigl(\begin{smallmatrix} 1 & \frac{u}{\ell^s} \\ 0 & 1 \end{smallmatrix} \bigr) \\
&= \sum_{u\,\mathrm{mod}\,\ell^s}\zeta^{uv} \left( \sum_{r\in\mathscr{R}}\lambda(\mathfrak{r})\nu(\mathfrak{r})(d^mf)|[r] \right)|[q_0]|\bigl(\begin{smallmatrix} 1 & \frac{u}{\ell^s} \\ 0 & 1 \end{smallmatrix} \bigr) \nonumber \\
&= \sum_{r\in\mathscr{R}}\lambda(\mathfrak{r})\nu(\mathfrak{r}) \psi(rq_0)(rq_0)^{k/2+m} d^m\left( \sum_{u\,\mathrm{mod}\,\ell^s}\zeta^{uv} f \| \bigl(\begin{smallmatrix} 1 & 0 \\ 0 & rq_0 \end{smallmatrix} \bigr)^{-1}\bigl(\begin{smallmatrix} 1 & -\frac{u}{\ell^s} \\ 0 & 1 \end{smallmatrix} \bigr) \right) \nonumber 
\end{align} }
Thus, the $\mathbf{q}$-expansion of $g_{\mathfrak{q}_0}$ is given by
\[ g_{\mathfrak{q}_0} (\mathbf{q}) = \sum_{r\in\mathscr{R}} \lambda(\mathfrak{r})\nu(\mathfrak{r}) \psi(rq_0)(rq_0)^{k+m} d^m \left( \sum_{n\geq 0} \left( \sum_{u\,\mathrm{mod}\,\ell^s}\zeta^{u(v-nrq_0)} \right)a(n,f)\mathbf{q}^{nrq_0}\right)\,.\]
and the Fourier coefficient $a(n,g_{\mathfrak{q}_0})$ is given by
\[ a(n,g_{\mathfrak{q}_0})=
\begin{cases}
  \ell^s \sum_{r\in\mathscr{R}} \lambda(\mathfrak{r})\nu(\mathfrak{r}) \psi(rq_0)(rq_0)^{k+m} n^m a(\frac{n}{rq_0},f) & \text{if } n\equiv v\; (\mathrm{mod}\; \ell^s),\\
  0 & \text{otherwise}.
\end{cases} \]
where we accept usual convention that $a(n,f)=0$ when $n$ is not an integer. 

Thus, if we fix a prime $r_0\in\mathscr{R}$, in order to contradict $g_{\mathfrak{q}}=0$ it suffices to find a prime $l$ outside $Np\ell$ and $\mathscr{R}$ such that
$l\equiv v(r_0q_0)^{-1}\; (\mathrm{mod}\,\ell^s)$ but $a(l,f)\not\equiv 0\; (\mathrm{mod}\,p)$. Indeed, then $l r_0q_0\equiv v \; (\mathrm{mod}\; \ell^s)$ and \[a(lr_0q_0,g_{\mathfrak{q}_0})= \ell^s\lambda\nu(\mathfrak{r}_0)\psi(r_0q_0) (r_0q_0)^{k+2m}l^ma(l,f)\not \equiv 0 \; (\mathrm{mod}\; p)\,.\] 

Let $\mathcal{P}$ be a prime ideal of $E$ above $p$. To do the outlined, recall that the celebrated constructions of Shimura ($k=2$), Deligne ($k>2$) and Deligne and Serre ($k=1$) attach to $f$ a Galois representation
\[\rho:\mathrm{Gal}(\bar{\mathbb{Q}}/\mathbb{Q})\to \mathrm{GL}_2(E_{\mathcal{P}})\]
such that for all primes $l\nmid Np$, $\rho$ is unramified at $l$ and for a Frobenius element $\mathrm{Fr}_l$ one has
\[ \mathrm{Tr} \rho(\mathrm{Fr}_l)=a(l,f)\text{ and }\mathrm{det}\rho(F_l)=\psi(l)l^{k-1}\,.\]
Let $\bar{\rho}$ be the reduction of $\rho$ modulo $\mathcal{P}$. Note that the fields $K=\bar{\mathbb{Q}}^{\mathrm{Ker}\bar{\rho}}$ and $\mathbb{Q}(\mu_{\ell^s})$ are linearly disjoint. Indeed, $\ell\nmid Np$ is unramified in the former while totally ramified in the latter. Then one can choose \[\sigma\in \mathrm{Gal}(K(\mu_{\ell^s})/\mathbb{Q})=\mathrm{Gal}(K/\mathbb{Q})\times \mathrm{Gal}(\mathbb{Q}(\mu_{\ell^s})/\mathbb{Q})\] 
such that 
\[\sigma|_K=\mathrm{id}_K \quad \text{and}\quad \sigma|_{\mathbb{Q}(\mu_{\ell^s})}=v(r_0q_0)^{-1}\in (\mathbb{Z}/\ell^s\mathbb{Z})^\times\cong \mathrm{Gal}(\mathbb{Q}(\mu_{\ell^s})/\mathbb{Q})\,.\]
Then the \v Cebotarev density theorem furnishes prime number $l$ such that $\mathrm{Fr}_l|_{K(\mu_{\ell^s})}=\sigma$, and we clearly have $a(l,f)=\mathrm{Tr} \rho(\mathrm{Fr}_l)=\mathrm{Tr}(\mathrm{id}_K) = 2\not\equiv 0 \;(\mathrm{mod}\;p)$ and $l\equiv  v(r_0q_0)^{-1} \; (\mathrm{mod}\,\ell^s)$ as desired.
\end{proof}
Note that we used condition $\ell\nmid N$ only to conclude that $K=\bar{\mathbb{Q}}^{\mathrm{Ker}\bar{\rho}}$ and $\mathbb{Q}(\mu_{\ell^s})$ are linearly disjoint.
(That being said, if we allow possibility of $\ell \mid N$, notation in Section \ref{sec:hewf} would have to be slightly adapted -- in place of $c_{ns}$ one would use its outside-of-$\ell$ part $c_{ns}^{(\ell)}$, however the argument remains valid.) Thus, the Theorem \ref{main} is valid under the following assumption milder than $\ell\nmid N$:
\begin{assumption}\label{assumption} $K=\bar{\mathbb{Q}}^{\mathrm{Ker}\bar{\rho}}$ and $\mathbb{Q}(\mu_{\ell^\infty})$ are linearly disjoint.
\end{assumption}
In particular, if $\ell|N$ but the modulo $\mathcal{P}$ reduction $\bar{\rho}$ happens to be unramified at $\ell$, this assumption is clearly satisfied. Assume for simplicity $\ell\| N$. If component $\pi_{\mathbf{f},\ell}=\pi(\mu_1,\mu_2)$ is a principal series, where $\mu_1$ and $\mu_2$ are characters of $\mathbb{Q}_\ell^\times$ with values in $E^\times$, one of the these characters is unramified and the other has conductor $\ell$. Write $\hat{\mu}_1$ and $\hat{\mu}_2$ for the characters of decomposition group $D_\ell\cong \mathrm{Gal}(\bar{\mathbb{Q}}_\ell/\mathbb{Q}_\ell)\to E_{\mathcal{P}}^\times$ corresponding to $\mu_1$ and $\mu_2$, respectively, by local class field theory. By the well known result of Carayol, we know that the semisimplification $\rho|_{D_\ell}$ is isomorphic to $\hat{\mu}_1\oplus \hat{\mu}_2$. Supposing that the ramified one among $\mu_1$ and $\mu_2$ has unramified reduction modulo $\mathcal{P}$ in order to obey Assumption \ref{assumption}, would bring $p\mid\ell-1$ placing $p$ inside prohibited set $\mathcal{S}(N,\ell)$ and we do not get desired result. 

We hope to treat this question in a future paper and here we limit ourselves with the following
\begin{spcl} \label{spcl}
Suppose that $\ell\| N$ and that $\pi_{\mathbf{f},\ell}$ is a special representation. Then Theorem \ref{main} holds.
\end{spcl}
\begin{proof}
Let $\pi_{\mathbf{f},\ell}=\mathrm{sp}(\alpha|\cdot|^{1/2},\alpha|\cdot|^{-1/2})$ for an unramified character $\alpha: \mathbb{Q}_\ell^\times \to E^{\times}$.
It is a well known result of Langlands (\cite{La73}) that $\rho|_{D_\ell}$ is isomorphic to 
\[\hat{\alpha}\otimes \bigl(\begin{smallmatrix} \chi_p^{\mathrm{cy}} & * \\ 0 & 1 \end{smallmatrix} \bigr)\]
where $\hat{\alpha}$ denotes the character of $D_\ell\cong \mathrm{Gal}(\bar{\mathbb{Q}}_\ell/\mathbb{Q}_\ell)\to E_{\mathcal{P}}^\times$ corresponding to $\alpha$ by local class field theory and $\chi_p^{\mathrm{cy}}$ is $p$-th cyclotomic character. Let $\mathbb{Q}_\ell^{\mathrm{ab}}$ denote maximal abelian extension of $\mathbb{Q}_\ell$ and $\mu_{(\ell),\infty}$ set of all roots of unity in $\bar{\mathbb{Q}}_\ell$ of order not divisible by $\ell$. Then for Artin local reciprocity map $(\,\cdot\,,\mathbb{Q}_\ell^{\mathrm{ab}}/\mathbb{Q}_\ell): \mathbb{Q}_\ell^\times \to \mathrm{Gal}(\mathbb{Q}_\ell^{\mathrm{ab}}/\mathbb{Q}_\ell)$
and $x=\ell^nu$ ($n\in \mathbb{Z}, u \in \mathbb{Z}_\ell^\times$), we have
\[ (x,\mathbb{Q}_\ell^{\mathrm{ab}}/\mathbb{Q}_\ell)=
\begin{cases}
\mathrm{Fr}_\ell^n & \text{on } \mathbb{Q}_\ell(\mu_{(\ell),\infty}),\\
 u^{-1} & \text{on } \mathbb{Q}_\ell(\mu_{\ell^\infty}).\\
\end{cases}
\]
Going back to proof of Theorem \ref{main}, it suffices to choose $\sigma\in K(\mu_{\ell^s})$ by first lifting $(v^{-1}r_0q_0,\mathbb{Q}_\ell^{\mathrm{ab}}/\mathbb{Q}_\ell)\in \mathrm{Gal}(\mathbb{Q}_\ell^{\mathrm{ab}}/\mathbb{Q}_\ell)$ to $D_\ell\cong \mathrm{Gal}(\bar{\mathbb{Q}}_\ell/\mathbb{Q}_\ell)$ and then projecting to $K(\mu_{\ell^s})$. Again, by virtue of  the \v Cebotarev density theorem we get prime number $l$ such that $\mathrm{Fr}_l|_{K(\mu_{\ell^s})}=\sigma$, and since both $\hat{\alpha}$ and $\chi_p^{\mathrm{cy}}$ are unramified at $\ell$, we conclude $a(l,f)=\mathrm{Tr} \rho(\mathrm{Fr}_l)=\mathrm{Tr}(\sigma) = 2\not\equiv 0 \;(\mathrm{mod}\;p)$ and $l\equiv v(r_0q_0)^{-1}\; (\mathrm{mod}\,\ell^s)$ as desired.
\end{proof}


\begin{thebibliography}{50}
\bibitem[Br10a]{Br2} M. Brako\v cevi\'c, \emph{Anticyclotomic $p$-adic L-function of central critical Rankin--Selberg L-value}, submitted preprint, 	arXiv:1007.5124v1, available at \url{http://arxiv.org/abs/1007.5124}
\bibitem[Br10b]{Br1} M. Brako\v cevi\'c, \emph{The Iwasawa $\mu$-invariant of anticyclotomic $p$-adic L-function}, in preparation
\bibitem[Ch]{Ch} C.-L. Chai, \emph{Families of ordinary abelian varieties: canonical coordinates, $p$-adic monodromy, Tate-linear subvarieties and Hecke orbits}, preprint, (2003), available at \url{http://www.math.upenn.edu/~chai} 
\bibitem[CoVa]{CoVa} C. Cornut, V. Vatsal, \emph{Nontriviality of Rankin--Selberg L-functions and CM points}, in ``L-functions and Galois Representations'', LMS Lecture Note Series 320, (2007), 121--186
\bibitem[De71]{De71} P. Deligne, \emph{Travaux de Shimura}, S\'em. Bourbaki, Exp. 389, Lecture Notes in Math. Vol. 244, Springer, (1971), 123--165
\bibitem[DeRa]{DeRa} P. Deligne, M. Rapoport, \emph{Les sch\'emas de modules de courbes elliptiques}, in ``Modular forms of one variable II'', Lecture Notes in Math. Vol. 349, Springer, (1973), 143--316
\bibitem[De79]{De79} P. Deligne, \emph{Vari\'et\'es de Shimura: interpr\'etation modulaire, et techniques de construction de mod\'eles canoniques}, Proc. Symp. Pure Math. 33.2 (1979), 247--290
\bibitem[Deu]{Deu} M. Deuring, \emph{Die Typen der Multiplikatorenringe elliptischer Funktionenk\"orper}, Abh.
Math. Sem. Hamburg 14 (1941), 197--272
\bibitem[Gr]{Gr} B. Gross, \emph{L-functions at the central critical point}, in ``Motives'', Proc. Symp. Pure Math. 55 (1994), 527--535
\bibitem[Hi88]{H88ajm} H. Hida, \emph{Modules of congruence of Hecke algebras and $L$-functions associated with cusp forms}, Amer. J. Math. 110 (1988), 323--382
\bibitem[Hi04]{HidaDwork} H. Hida,  \emph{Non-vanishing modulo p of Hecke L-values}, in ``Geometric Aspects of Dwork Theory'', Walter de Gruyter, Berlin, (2004), 735--784
\bibitem[PAF]{PAF} H. Hida, \emph{p-Adic Automorphic Forms on Shimura Varieties}, Springer Monographs in Mathematics, (2004), Springer
\bibitem[Hi07]{HidaLMS} H. Hida,  \emph{Non-vanishing modulo p of Hecke L-values and application}, in ``L-functions and Galois Representations'', LMS Lecture Note Series 320, (2007), 207--269
\bibitem[Hi10a]{minv} H. Hida, \emph{The Iwasawa $\mu$-Invariant of $p$-Adic Hecke $L$-functions}, Ann. of Math. 172 (2010), 41--137
\bibitem[Hi10b]{HidaNV} H. Hida, \emph{Central critical values of modular Hecke L-functions},  Kyoto Journal of Mathematics, 50 (2010), 777--826
\bibitem[Ja]{Ja} H. Jacquet, \emph{Automorphic forms on $\mathrm{GL}(2)$. Part II}, Lecture Notes in Math., Vol. 278, Springer, (1972) 
\bibitem[JaLa]{JaLa} H. Jacquet, R.P. Langlands, \emph{Automorphic forms on $\mathrm{GL}(2)$}, Lecture Notes in Math., Vol. 114, Springer, (1970) 
\bibitem[Ka76]{Ka76} N.M. Katz, \emph{p-adic Interpolation of Real Analytic Eisenstein Series}, Ann. of Math. 104 (1976), 459--571
\bibitem[Ka78]{Ka} N.M. Katz, \emph{p-adic L-functions for CM fields}, Invent. Math. 49 (1978), 199--297 
\bibitem[Ko]{Ko} R. Kottwitz, \emph{Points on Shimura varieties over finite fields}, J. Amer. Math. Soc. 5 (1992), 373--444
\bibitem[La73]{La73} R.P. Langlands, Modular forms and $\ell$-adic representations, Lecture Notes in Math., Vol. 349,
Springer, (1973), 361--500
\bibitem[CRT]{CRT} H. Matsumura, \emph{Commutative Ring Theory}, Cambridge Studies in Advanced Mathematics 8, Cambridge Univ. Press, (1986) 
\bibitem[Mi]{Mi} J. Milne, \emph{Canonical Models of (Mixed) Shimura Varieties and Automorphic Vector Bundles}, Perspective Math., 10, (1990), 283--414
\bibitem[ABV]{ABV} D. Mumford, \emph{Abelian Varieties}, TIFR Studies in Mathematics, Oxford University Press, (1994)
\bibitem[SeTa]{SeTa} J.-P. Serre, J. Tate, \emph{Good reduction of abelian varieties}, Ann. of Math. 88 (1968), 452--517
\bibitem[Si]{Si} W. Sinnott, \emph{On a theorem of L. Washington}, Ast\'erisque 147-148 (1987), 209--224
\bibitem[Sh66]{Sh66} G. Shimura, \emph{Moduli and fibre system of abelian varieties}, Ann. of Math. 83 (1966), 294--338
\bibitem[IAT]{IAT} G. Shimura, \emph{Introduction to the Arithmetic Theory of Automorphic Functions}, Princeton University Press and Iwanami Shoten, (1971), Princeton--Tokyo
\bibitem[Sh75]{Sh75} G. Shimura, \emph{On some arithmetic properties of modular forms of one and several variables}, Ann. of Math. 102 (1975), 491--515
\bibitem[Sh76]{Sh76} G. Shimura, \emph{The special values of the zeta functions associated with cusp forms}, Comm. Pure Appl. Math. 29 (1976), no. 6, 783--804
\bibitem[ACM]{ACM} G. Shimura, \emph{Abelian Varieties with Complex Multiplication and Modular Functions}, Princeton University Press, (1998)
\bibitem[AAF]{AAF} G. Shimura, \emph{Arithmeticity in the Theory of Automorphic Forms}, Mathematical Surveys and Monographs 82, AMS, (2000)
\bibitem[Sun]{HS} H. Sun, \emph{Non-vanishing $\mathrm{mod}\,p$ of special L-values}, UCLA PhD Thesis, (2007), available at \url{http://newton.kias.re.kr/~haesang/}
\bibitem[Ta]{Ta} J. Tate, \emph{Fourier Analysis in Number Fields and Hecke's Zeta-Functions}, PhD Thesis, Princeton University, reprinted in ``Algebraic number theory'' ed. J.W.S. Cassels and A. Fr\"ohlich, Academic Press, London, (1967), 305--347
\bibitem[Va02]{Va02} V. Vatsal, \emph{Uniform distribution of Heegner points}, Invent. Math. 148 (2002), 1--46
\bibitem[Va03]{Va03} V. Vatsal, \emph{Special values of anticyclotomic L-functions}, Duke Math. J. 116 (2003), no. 2, 219--261
\bibitem[Va06]{VaICM} V. Vatsal, \emph{Special values of L-functions modulo $p$}, International Congress of Mathematicians. Vol. II, Eur. Math. Soc., Z\"urich, (2006), 501--514
\bibitem[Wa]{Wa} J.-L. Waldspurger, \emph{Sur les valeurs de certaines fonctions L automorphes en leur centre de sym\'etrie}, Compositio Math. 54 (1985), 173--242
\bibitem[Zh]{Zh} S. Zhang, \emph{Gross-Zagier formula for $\mathrm{GL}(2)$} , Asian J. Math. 5 (2001), no. 2., 183--290
\end{thebibliography}
\end{document}